\documentclass[11pt,a4paper,reqno]{amsart}
\usepackage[T1]{fontenc}
\usepackage[utf8]{inputenc}
\usepackage{lmodern,microtype,amsmath,amssymb,mathtools,mathrsfs}
\usepackage[a4paper,left=28mm,right=28mm,top=27mm,bottom=27mm,headheight=14pt,headsep=7mm,footskip=11mm]{geometry}
\usepackage{enumitem,needspace}
\usepackage[hidelinks,unicode]{hyperref}
\hypersetup{pdftitle={The Rademacher quotient of L1 embeds into L1 and its weak Calkin algebra is uniformly incompressible},pdfauthor={Amir Bahman Nasseri},pdfsubject={Rectangle-difference embedding, weak essential norms, and uniform incompressibility}}
\allowdisplaybreaks[2]
\setlist[enumerate]{label=\textup{(\roman*)},leftmargin=2em,itemsep=3pt}
\newtheorem{maintheorem}{Theorem}

\newtheorem{theorem}{Theorem}[section]
\newtheorem{proposition}[theorem]{Proposition}
\newtheorem{lemma}[theorem]{Lemma}
\newtheorem{corollary}[theorem]{Corollary}
\theoremstyle{definition}

\theoremstyle{remark}
\newtheorem{remark}[theorem]{Remark}
\newtheorem{example}[theorem]{Example}
\newcommand{\N}{\mathbb N}
\newcommand{\E}{\mathbb E}
\newcommand{\PP}{\mathbb P}
\newcommand{\Prob}{\mathbb P}
\newcommand{\A}{\mathcal A}
\newcommand{\F}{\mathcal F}
\newcommand{\B}{\mathscr B}
\newcommand{\W}{\mathscr W}
\newcommand{\Comp}{\mathscr K}
\newcommand{\Ap}{\mathscr A}
\newcommand{\Qw}{\mathfrak Q_{\!w}}
\newcommand{\co}{\mathrm{co}}
\newcommand{\Id}{\mathrm I}
\newcommand{\one}{\mathbf 1}
\newcommand{\dd}{\,d}
\DeclareMathOperator{\dist}{dist}
\DeclareMathOperator{\spanop}{span}
\DeclareMathOperator{\ran}{ran}
\DeclareMathOperator{\esssup}{ess\,sup}

\newcommand{\norm}[1]{\lVert #1\rVert}
\newcommand{\abs}[1]{\lvert #1\rvert}
\newcommand{\pair}[2]{\langle #1,#2\rangle}

\title[The Rademacher quotient and its weak Calkin algebra]{The Rademacher quotient of $L^1$ embeds into $L^1$ and its weak Calkin algebra is uniformly incompressible}
\author{Amir Bahman Nasseri}
\date{19 September 2026; revised 22 September 2026}
\begin{document}
\begin{abstract}
Let $R$ be the closed linear span of the Rademacher functions in $L^1(0,1)$.
We prove that $L^1(0,1)/R$ embeds isomorphically into $L^1(0,1)$, answering affirmatively a question of Gonz\'alez and Mart\'inez-Abej\'on \cite[p.~140]{GM97}.
We then use this embedding to show that the weak Calkin algebra of the quotient is uniformly incompressible.
\end{abstract}
\maketitle

\section*{Introduction and main results}
Quotients of $L^1$ by reflexive subspaces need not be isomorphic to subspaces of $L^1$; a counterexample with a subspace isomorphic to $\ell_2$ is given in \cite[Proposition~1]{GM97}.
The position of the reflexive subspace is therefore essential.
Gonz\'alez and Mart\'inez-Abej\'on asked whether the quotient by the closed span of the Rademacher functions embeds into $L^1(0,1)$; see \cite[p.~140]{GM97}.
We address that question by an explicit construction and use the resulting embedding to study the weak Calkin algebra of the quotient.

Throughout, all Banach spaces are real and all operators are bounded and linear.
Let
\[
 \Delta\coloneqq\{-1,1\}^{\N},\qquad
 \Prob\coloneqq\bigotimes_{j\geq1}\tfrac12(\delta_{-1}+\delta_1),\qquad
 E\coloneqq L^1(\Delta,\Prob),\quad r_j(x)\coloneqq x_j,
\]
and set
\begin{equation}\label{eq:spaces}
 R\coloneqq\overline{\spanop\{r_j:j\geq1\}}^{\,L^1},\qquad
 \A\coloneqq\mathbb R\one+R,\qquad X\coloneqq E/R.
\end{equation}
The metric quotient map is denoted by $q:E\to X$.
Constants are not included in $R$.
The usual binary-digit identification carries this pair to the classical Rademacher pair on $(0,1)$.
Every element of $R$ has mean zero, so $\A$ is closed: convergence of $c_k\one+u_k$ first gives convergence of $c_k$ by taking means, and then convergence of $u_k\in R$.

For independent $x,y,\sigma\in\Delta$, define $z,w\in\Delta$ by
\begin{equation}\label{eq:swaps}
 z_j\coloneqq\frac{1+\sigma_j}{2}x_j+\frac{1-\sigma_j}{2}y_j,
 \qquad
 w_j\coloneqq\frac{1-\sigma_j}{2}x_j+\frac{1+\sigma_j}{2}y_j.
\end{equation}
Our rectangle-difference operator is
\begin{equation}\label{eq:D-infinite}
 Df(x,y,\sigma)\coloneqq f(x)+f(y)-f(z)-f(w).
\end{equation}
Each of the four points has distribution $\Prob$; in particular $D:E\to L^1(\Prob^{\otimes3})$ is well defined.
The finite-dimensional version of this defect is exactly the four-query direct-sum defect of Bogdanov and Prakriya \cite[Algorithm~1]{BP21}, after identifying $\{-1,1\}$ with $\{0,1\}$.
The same query pattern was proposed for $\mathbb F_2$-valued functions by Dinur and Golubev \cite[Test~6]{DG19}.
On a two-point product, scalar-valued direct sums are precisely the functions in $\A_n$ defined below.
The soundness theorem in \cite[Theorem~1]{BP21} controls Hamming distance by the probability that the defect is nonzero.
We instead establish a dimension-free estimate of the $L^1$ distance by the integral of the absolute defect.
For a fixed bipartition of the variables, rectangle differences and $L^p$ approximation estimates were already considered by Bogdanov and Wang \cite[Section~3, Claims~5--6]{BW20}.
There the approximants are sums of arbitrary functions of the two coordinate blocks; here the required approximants are affine Rademacher functions, with a constant uniform in the number of coordinates.
Thus the defect and the probabilistic estimates used below have precedents in the literature; the estimate to be proved is their dimension-free $L^1$ quotient application.

\begin{maintheorem}[The embedding]\label{thm:main}
There is an absolute constant $C\geq1$ such that
\begin{equation}\label{eq:main-distance}
 C^{-1}\dist_{L^1(\Prob)}(f,\A)
 \leq \norm{Df}_{L^1(\Prob^{\otimes3})}
 \leq 4\dist_{L^1(\Prob)}(f,\A)
 \qquad(f\in L^1(\Prob)).
\end{equation}
Consequently $\ker D=\A$, and the map
\begin{equation}\label{eq:main-J}
 J:L^1(\Prob)/R\longrightarrow
 \mathbb R\oplus_1 L^1(\Prob^{\otimes3}),
 \qquad J(f+R)\coloneqq(\E f,Df),
\end{equation}
is well defined and satisfies
\begin{equation}\label{eq:main-J-bounds}
 \frac{1}{2C}\norm{f+R}_{L^1/R}
 \leq \norm{J(f+R)}
 \leq 5\norm{f+R}_{L^1/R}.
\end{equation}
Both $L^1(\Prob)/R$ and $L^1(\Prob)/\A$ therefore embed into $L^1(0,1)$. One admissible value of $C$ is given in \eqref{eq:constants} below.
\end{maintheorem}
For Banach spaces $Y,Z$, write $\B(Y,Z)$ and $\W(Y,Z)$ for the bounded and weakly compact operators and put
\[
 \norm{A}_w=\inf_{W\in\W(Y,Z)}\norm{A-W},\qquad
 Y^{\co}=Y^{**}/J_Y Y,
\]
where $J_Y:Y\to Y^{**}$ is the canonical embedding.
The residuum operator is
$A^{\co}(y^{**}+J_YY)=A^{**}y^{**}+J_ZZ$.
The weak Calkin algebra is $\Qw(Y)=\B(Y)/\W(Y)$, with $\norm{[T]}=\norm{T}_w$.
A normed algebra is incompressible if every bounded injective homomorphism into a normed algebra is bounded below; it is uniformly incompressible if a single positive lower bound works for all contractive injective homomorphisms.
Homomorphisms need not preserve an identity, and algebra norms are submultiplicative.

\begin{maintheorem}[The weak essential norm]\label{c:thm:mainnorm}
For every $T\in\B(X)$,
\begin{equation}\label{c:eq:mainnorm}
 \norm{T^{\co}}\leq\norm{T}_w\leq2\norm{T^{\co}}.
\end{equation}
Thus $[T]\mapsto T^{\co}$ is an injective homomorphism with closed range and inverse norm at most $2$ on its image.
\end{maintheorem}

\begin{maintheorem}[Uniform incompressibility]\label{c:thm:mainalgebra}
The algebra $\Qw(X)$ is uniformly incompressible.
Let $C$ be the constant in Theorem~\ref{thm:main}.
For every normed algebra $B$ and every bounded injective homomorphism $\Phi:\Qw(X)\to B$,
\begin{equation}\label{c:eq:mainalgebra}
 \norm{\Phi(a)}\geq
 \frac{\norm{a}}{40C\norm{\Phi}^2}
 \qquad(a\in\Qw(X)).
\end{equation}
\end{maintheorem}

The proof is organized into two sections.
Section~\ref{sec:embedding} proves Theorem~\ref{thm:main}.
A Walsh multiplier computation yields a dimension-free $L^{5/4}$ estimate.
The probabilistic input is \emph{hypercontractivity}: a suitable averaging operator sends $L^p$ into $L^q$, for $q>p$, with norm at most one.
Here the noise operators $T_\Theta$ defined in \eqref{eq:noise-definition} satisfy
$\norm{T_\Theta f}_2\leq\norm{f}_p$ when $1<p\leq2$ and $0\leq\Theta\leq\sqrt{p-1}$.
This stronger integrability after averaging controls the simultaneous occurrence of large values at two correlated points.
A truncation argument then gives the $L^1$ quotient estimate.
All probabilistic inputs used in this passage are formulated below; the hypercontractive input is also proved.

Section~\ref{sec:algebra} proves Theorems~\ref{c:thm:mainnorm} and~\ref{c:thm:mainalgebra}.
First, a Rademacher restriction with a weakly compact extension has an exact approximable extension, whose norm can be made arbitrarily close to that of any prescribed bounded extension.
This corrects weakly compact approximations so that they descend to the quotient and gives Theorem~\ref{c:thm:mainnorm}.
Theorem~\ref{thm:main} then allows the quantitative $AL$-factorization theorem of Acuaviva--Nasseri \cite[Theorem~1.4]{AN} to be applied to $DTq$ for a suitable embedding $D$.
A fixed nonzero idempotent converts these factorizations into Theorem~\ref{c:thm:mainalgebra}.
The same argument gives a retraction-based extension to larger $L^1$-spaces.

\section{The Rademacher quotient as a subspace of \texorpdfstring{$L^1$}{L1}}\label{sec:embedding}
\subsection{Probabilistic tools}\label{subsec:probability-tools}

All probability measures below have total mass one. We write $\E Z\coloneqq\int Z\,dP$ and $\PP(A)\coloneqq P(A)$. We use the standard integration theorems and record the probabilistic tools needed in the proof.

\begin{proposition}[Conditional expectation]\label{prop:integration-tools}
Let $(S,\Sigma,P)$ be a probability space and let $\mathcal G\subseteq\Sigma$ be a sub-$\sigma$-algebra.
For $h\in L^1(P)$, there is a unique $\mathcal G$-measurable element of $L^1(P)$, denoted by $\E[h\mid\mathcal G]$, such that
\[
 \int_A\E[h\mid\mathcal G]\,dP=\int_Ah\,dP
 \qquad(A\in\mathcal G).
\]
Uniqueness is understood modulo null sets. Conditional expectation is linear, positive, preserves integrals, and fixes $\mathcal G$-measurable functions.
For $h\in L^p(P)$, $1\leq p<\infty$, conditional Jensen's inequality gives
\[
 |\E[h\mid\mathcal G]|^p\leq\E[|h|^p\mid\mathcal G]
 \quad\text{almost everywhere}.
\]
Consequently $\norm{\E[h\mid\mathcal G]}_p\leq\norm{h}_p$ for $1\leq p\leq\infty$.
On a product probability space, conditioning on one factor amounts to integrating over the other factor.
\end{proposition}
For the definition and contractivity of conditional expectation, see \cite[Definition~6.1.1 and Lemma~6.1.2]{AK16}; for conditional Jensen's inequality, see \cite[Theorem~4.1.10]{Dur19}.

\begin{lemma}[The Bernoulli product law]\label{lem:Bernoulli-law}
There is a probability measure $\Prob$ on $\Delta=\{-1,1\}^{\mathbb N}$, with its product $\sigma$-algebra, such that
\[
 \Prob\{x:x_j=\varepsilon_j\ (j\in F)\}=2^{-|F|}
 \quad(F\subset\mathbb N\text{ finite},\ \varepsilon_j\in\{-1,1\}).
\]
Its coordinate functions are independent fair signs. More generally, measurable functions depending on disjoint coordinate sets are independent.
Explicitly, if $g:\{-1,1\}^n\to\mathbb R$ and $f(x)\coloneqq g(x_1,\ldots,x_n)$, then
\[
 \int_\Delta f(x)\,d\Prob(x)
 =2^{-n}\sum_{\varepsilon\in\{-1,1\}^n}g(\varepsilon).
\]
This is the Bernoulli instance of the product-measure construction; see \cite[Theorem~A.3.1, Example~A.3.2, and Theorems~2.1.11--2.1.13]{Dur19}.
\end{lemma}
\begin{proof}
Push Lebesgue measure on $[0,1)$ forward under the binary-digit map $t\mapsto((-1)^{\beta_j(t)})_j$. Dyadic intervals show that the first $n$ signs have the uniform distribution; summing over unspecified signs gives the displayed formula. Independence follows first for cylinders and then for the generated $\sigma$-algebras by the monotone-class theorem.
\end{proof}

\begin{lemma}[Distribution formulas]\label{lem:distribution-tools}
Let $Z\geq0$ be integrable and put $a(t)\coloneqq\PP(Z>t)$.
Markov's inequality and the layer-cake formula \cite[Theorem~1.6.4 and Exercise~1.7.2]{Dur19} give, for $t,M>0$,
\[
 a(t)\leq\frac{\E Z}{t},\qquad
 \E Z=\int_0^\infty a(t)\,dt,\qquad
 \E[Z\mathbf1_{\{Z>M\}}]=Ma(M)+\int_M^\infty a(t)\,dt.
\]
\end{lemma}
\begin{proof}
The first assertion follows from $t\mathbf1_{\{Z>t\}}\leq Z$. For the others, integrate the pointwise identities
\[
 Z=\int_0^\infty\mathbf1_{\{Z>t\}}\,dt,\qquad
 Z\mathbf1_{\{Z>M\}}=M\mathbf1_{\{Z>M\}}+
 \int_M^\infty\mathbf1_{\{Z>t\}}\,dt
\]
using Tonelli's theorem. We also use the finite union bound in its pointwise form $\mathbf1_{\cup_j A_j}\leq\sum_j\mathbf1_{A_j}$.
\end{proof}

\begin{samepage}
For probability spaces $(S,\Sigma,P)$ and $(T,\mathcal T,Q)$, with variables $x\in S$ and $t\in T$, we write $L_x^p\coloneqq L^p(S,P)$ and $L_t^q\coloneqq L^q(T,Q)$.
For a measurable $h:S\times T\to\mathbb R$ and $1\leq p,q<\infty$, define
\begin{align*}
 \norm{h}_{L_t^q(L_x^p)}
 &\coloneqq\left(\int_T\left(\int_S|h(x,t)|^p\,dP(x)\right)^{q/p}\,dQ(t)\right)^{1/q},\\
 \norm{h}_{L_x^p(L_t^q)}
 &\coloneqq\left(\int_S\left(\int_T|h(x,t)|^q\,dQ(t)\right)^{p/q}\,dP(x)\right)^{1/p}.
\end{align*}
The subscripts identify the integration variables. In $L_t^q(L_x^p)$, one first takes the $L^p$ norm in $x$ for each fixed $t$, and then the $L^q$ norm of the resulting function of $t$.
The next inequality is the standard mixed-norm form of Minkowski's integral inequality; compare \cite[Chapter~III, Lemma~1]{Bon70} and \cite[Lemma~2]{Bec75}. We include its proof for the reader's convenience, together with the elementary interpolation estimates used later.

\end{samepage}

\begin{lemma}[Mixed norms and interpolation]\label{lem:mixed-interpolation}
For $1\leq p\leq q<\infty$ and measurable $h$ on a product of probability spaces,
\[
 \norm{h}_{L^q_t(L^p_x)}\leq\norm{h}_{L^p_x(L^q_t)}.
\]
On a probability space, $\norm{h}_p\leq\norm{h}_q$ for $p\leq q$. The further estimates used below are
\[
 \norm{h}_{5/4}\leq\norm{h}_\infty^{1/5}\norm{h}_1^{4/5},
 \qquad
 \norm{h}_2\leq\norm{h}_1^{1/3}\norm{h}_4^{2/3}.
\]
\end{lemma}
\begin{proof}
Put $s=q/p\geq1$. The integral triangle inequality in $L^s_t$ gives
\[
 \norm{h}_{L^q_t(L^p_x)}^p
 =\left\lVert\int |h(x,\cdot)|^p\,dP(x)\right\rVert_{L^s_t}
 \leq\int\norm{|h(x,\cdot)|^p}_{L^s_t}\,dP(x)
 =\norm{h}_{L^p_x(L^q_t)}^p.
\]
The norm comparison follows from H\"older's inequality and $P(S)=1$. Integrating $|h|^{5/4}\leq\norm{h}_\infty^{1/4}|h|$ proves the first interpolation estimate. Applying H\"older with exponents $3/2$ and $3$ to $|h|^2=|h|^{2/3}|h|^{4/3}$ proves the second.
\end{proof}
\subsection{Finite cubes and the normal operator}\label{sec:finite}

For $n\geq1$, let
\[
 \Delta_n\coloneqq\{-1,1\}^n,\qquad
 \Prob_n(\{x\})\coloneqq2^{-n},\qquad
 \A_n\coloneqq\spanop\{\one,r_1,\ldots,r_n\},\qquad Q_n\coloneqq\Prob_n^{\otimes3}.
\]
Every norm on $\Delta_n$ uses the uniform probability measure $\Prob_n$.
A norm of $D_nf$ is taken on $(\Delta_n^3,Q_n)$.
Probabilities and expectations involving the maps $U_i$ below are taken with respect to $Q_n$; we write $\E_{Q_n}Z\coloneqq\int Z\,dQ_n$.

Let $x,y,\sigma$ be independent with distribution $\Prob_n$. Define $z,w$ by \eqref{eq:swaps} for $1\leq j\leq n$, and define $D_n$ by the finite version of \eqref{eq:D-infinite}. It is convenient to put
\[
 U_0\coloneqq x,\quad U_1\coloneqq y,\quad U_2\coloneqq z,\quad U_3\coloneqq w,
 \qquad (s_0,s_1,s_2,s_3)\coloneqq(1,1,-1,-1).
\]
Thus $D_nf=\sum_{i=0}^3s_i f(U_i)$.

\Needspace{18\baselineskip}
\begin{samepage}
\subsubsection{Marginals, conditional expectations, and noise}

\begin{lemma}[Laws of the four points]\label{lem:coupling}
Let $U_i:\Delta_n^3\to\Delta_n$, $0\leq i\leq3$, be the maps defined above on the probability space $(\Delta_n^3,Q_n)$.
For every $u,v\in\Delta_n$, the following identities hold:
\begin{enumerate}
\item Each marginal is uniform:
\[
 Q_n(U_i=u)=2^{-n}\qquad(0\leq i\leq3).
\]
\item The pairs $(U_0,U_1)$ and $(U_2,U_3)$ are independent:
\[
 Q_n(U_i=u,U_j=v)=2^{-2n}
 \qquad((i,j)\in\{(0,1),(2,3)\}).
\]
\item For $(i,j)\in\{(0,2),(0,3),(1,2),(1,3)\}$,
\[
 Q_n(U_j=v\mid U_i=u)
 =\prod_{k=1}^n\left(\tfrac34\mathbf1_{\{v_k=u_k\}}+
                     \tfrac14\mathbf1_{\{v_k=-u_k\}}\right).
\]
Thus, conditionally on $U_i=u$, the coordinates of $U_j$ are independent, with $Q_n((U_j)_k=u_k\mid U_i=u)=3/4$.
The same formula holds with $i$ and $j$ interchanged.
\end{enumerate}
\end{lemma}
\end{samepage}
\begin{proof}
For a fixed $\sigma$, the transformation $(x,y)\mapsto(z,w)$ permutes the two entries in each coordinate and preserves $\Prob_n\otimes\Prob_n$. This proves the assertion about $(z,w)$ and hence all marginal assertions. The independence of $(x,y)$ is part of the construction.

For example, conditionally on $x$, the variable $z_j$ equals $x_j$ if $\sigma_j=1$; if $\sigma_j=-1$, it is the independent fair sign $y_j$. Therefore
\[
 Q_n(z_j=x_j\mid x)=\tfrac12+\tfrac12\cdot\tfrac12=\tfrac34.
\]
The pairs $(y_j,\sigma_j)$ are independent over $j$, giving the product conditional law. Interchanging $x,y$ or $z,w$ proves the other three cases. The product on the right-hand side of (iii) is symmetric in the two points. Together with their equal uniform marginals, this proves the same conditional formula with the two points interchanged.
\end{proof}

For $0\leq\Theta\leq1$, define the probability measure
\[
 \nu_{n,\Theta}\coloneqq
 \bigotimes_{j=1}^n\left(\frac{1+\Theta}{2}\delta_1+
                         \frac{1-\Theta}{2}\delta_{-1}\right),
 \qquad u\odot\epsilon\coloneqq(u_1\epsilon_1,\ldots,u_n\epsilon_n).
\]
The noise operator on $\Delta_n$ is
\begin{equation}\label{eq:noise-definition}
 (T_\Theta f)(u)\coloneqq
 \int_{\Delta_n} f(u\odot\epsilon)\,d\nu_{n,\Theta}(\epsilon).
\end{equation}
Here $\epsilon$ is the integration variable, and $\nu_{n,\Theta}$ is used only in this averaging formula; the $L^p$ norms still use $\Prob_n$.
For every fixed $\epsilon$, the map $u\mapsto u\odot\epsilon$ permutes $\Delta_n$ and preserves $\Prob_n$.
Since $\nu_{n,\Theta}$ is a probability measure, \eqref{eq:noise-definition} gives
\[
 f\geq0\ \Longrightarrow\ (T_\Theta f)(u)\geq0,
 \qquad (T_\Theta\one)(u)=\int_{\Delta_n}1\,d\nu_{n,\Theta}=1.
\]
These are positivity and the identity $T_\Theta\one=\one$.
The permutation property and Fubini's theorem also give
\[
 \int_{\Delta_n}T_\Theta f\,d\Prob_n
 =\int_{\Delta_n}f\,d\Prob_n.
\]
Jensen's inequality gives
\[
 |T_\Theta f|^p\leq T_\Theta(|f|^p),\qquad
 \int_{\Delta_n}T_\Theta(|f|^p)\,d\Prob_n
 =\int_{\Delta_n}|f|^p\,d\Prob_n \quad(1\leq p<\infty).
\]
Thus $\norm{T_\Theta f}_p\leq\norm{f}_p$; the $p=\infty$ assertion follows directly from \eqref{eq:noise-definition}.
Define the mean projection by $P_0f\coloneqq(\int f\,d\Prob_n)\one$.
The identity $T_0=P_0$ follows because $\nu_{n,0}=\Prob_n$ and $\epsilon\mapsto u\odot\epsilon$ is a bijection:
\[
 (T_0f)(u)=2^{-n}\sum_{\epsilon\in\Delta_n}f(u\odot\epsilon)
 =2^{-n}\sum_{v\in\Delta_n}f(v)=(P_0f)(u).
\]
The identity $T_1=I$ follows because $\nu_{n,1}$ is concentrated at $(1,\ldots,1)$:
\[
 (T_1f)(u)=f(u\odot(1,\ldots,1))=f(u).
\]

For $S\subseteq\{1,\ldots,n\}$ define the Walsh function
\[
 w_S(u)\coloneqq\prod_{j\in S}u_j,\qquad w_\varnothing\coloneqq\one.
\]
The Walsh basis and the following noise diagonalization are classical; see \cite[Theorem~1.5 and Proposition~2.47]{ODonnell}.
We give the calculation explicitly.
\begin{samepage}
For $S,B\subseteq\{1,\ldots,n\}$ with $S\ne B$, independence of the fair coordinates and the presence of a coordinate in the symmetric difference $S\mathbin\triangle B$ give
\[
 \int_{\Delta_n}w_Sw_B\,d\Prob_n
 =\int_{\Delta_n}\prod_{j\in S\mathbin\triangle B}u_j\,d\Prob_n=0.
\]
If $S=B$, the integral is $1$.
\end{samepage}
There are $2^n$ Walsh functions and $\dim L^2(\Prob_n)=2^n$, so they form an orthonormal basis.
Furthermore, the product definition of $\nu_{n,\Theta}$ gives
\begin{align*}
 (T_\Theta w_S)(u)
 &=w_S(u)\int_{\Delta_n}\prod_{j\in S}\epsilon_j\,d\nu_{n,\Theta}(\epsilon)\\
 &=w_S(u)\prod_{j\in S}\left(\frac{1+\Theta}{2}-\frac{1-\Theta}{2}\right)
 =\Theta^{|S|}w_S(u).
\end{align*}
The empty product is $1$, including when $\Theta=0$.
Hence
\begin{equation}\label{eq:noise-spectrum}
 T_\Theta w_S=\Theta^{|S|}w_S,
 \qquad T_{\Theta_1}T_{\Theta_2}=T_{\Theta_1\Theta_2}
 \quad(0\leq\Theta_1,\Theta_2\leq1).
\end{equation}
The second identity holds because the two operators have the same multiplier on every basis vector.
The multipliers are real, so $T_\Theta$ is self-adjoint on $L^2(\Prob_n)$.
Thus the $w_S$ are eigenfunctions, and the numbers $\Theta^{|S|}$ are the corresponding eigenvalues.
We write $T\coloneqq T_{1/2}$ when $n$ is fixed.

\begin{lemma}\label{lem:adjoint}
For $1\leq p\leq\infty$ and $f\in L^p(\Delta_n,\Prob_n)$,
\begin{equation}\label{eq:D-bounded}
 \norm{D_nf}_p\leq4\norm{f}_p,
 \qquad D_na=0\quad(a\in\A_n).
\end{equation}
Let $D_n^\dagger$ denote the Hilbert-space adjoint for the normalized $L^2$ inner products. Its finite-dimensional formula defines an operator from $L^p(\Delta_n^3,Q_n)$ to $L^p(\Delta_n,\Prob_n)$ and, for every $v\in L^p(\Delta_n^3,Q_n)$, satisfies
\begin{equation}\label{eq:adjoint-bound}
 \norm{D_n^\dagger v}_{L^p(\Prob_n)}
 \leq4\norm{v}_{L^p(\Prob_n^{\otimes3})}.
\end{equation}
Moreover,
\begin{equation}\label{eq:normal}
 D_n^\dagger D_n=4(I+P_0-2T).
\end{equation}
\end{lemma}
\begin{proof}
Fix $1\leq p\leq\infty$. For $0\leq i\leq3$, define
\[
 C_i:L^p(\Delta_n,\Prob_n)\longrightarrow L^p(\Delta_n^3,Q_n),
 \qquad C_i f\coloneqq f\circ U_i.
\]
By Lemma~\ref{lem:coupling}, $\norm{C_i f}_{L^p(Q_n)}=\norm{f}_{L^p(\Prob_n)}$ for every $f\in L^p(\Delta_n,\Prob_n)$.
Consequently
\[
 D_n=C_0+C_1-C_2-C_3,\qquad
 \norm{D_nf}_{L^p(Q_n)}\leq\sum_{i=0}^3\norm{C_i f}_{L^p(Q_n)}
 =4\norm{f}_{L^p(\Prob_n)}.
\]
The identities $D_n\one=0$ and $D_nr_j=0$ follow from $x_j+y_j=z_j+w_j$.

Define
\[
 F_i:L^p(\Delta_n^3,Q_n)\longrightarrow L^p(\Delta_n,\Prob_n)
\]
by the following formula, for $v\in L^p(\Delta_n^3,Q_n)$:
\[
 (F_i v)(u)\coloneqq
 2^{-2n}\sum_{\substack{\omega\in\Delta_n^3\\U_i(\omega)=u}}v(\omega)
 =\E_{Q_n}[v\mid U_i=u]\qquad(u\in\Delta_n).
\]
The fiber has $2^{2n}$ elements because $U_i$ is uniform. Since $U_i$ has distribution $\Prob_n$, Jensen's inequality gives $\norm{F_i v}_p\leq\norm{v}_p$ for $1\leq p<\infty$; the supremum-norm estimate follows directly from the averaging formula. The normalized inner products are
\[
 \pair{f}{g}_{L^2(\Prob_n)}\coloneqq\int_{\Delta_n}fg\,d\Prob_n,
 \qquad
 \pair{v}{w}_{L^2(Q_n)}\coloneqq\int_{\Delta_n^3}vw\,dQ_n.
\]
For $f\in L^2(\Prob_n)$ and $v\in L^2(Q_n)$, direct integration gives
\[
 \pair{C_i f}{v}_{L^2(\Prob_n^{\otimes3})}
 =\pair{f}{F_i v}_{L^2(\Prob_n)}.
\]
Thus $D_n^\dagger=F_0+F_1-F_2-F_3$, proving \eqref{eq:adjoint-bound}.

To compute the normal operator, expand
\[
 D_n^\dagger D_n=\sum_{i,j=0}^3s_i s_j F_i C_j.
\]
The four terms with $i=j$ give $4I$. The four ordered pairs obtained from $(0,1)$ and $(2,3)$ give $4P_0$, since those points are independent and $s_i=s_j$. Each of the eight remaining ordered pairs gives $-T_{1/2}$ by Lemma~\ref{lem:coupling}. This proves \eqref{eq:normal}.
\end{proof}

\begin{corollary}\label{cor:finite-kernel}
One has $\ker D_n=\A_n$. More precisely,
\begin{equation}\label{eq:normal-spectrum}
 D_n^\dagger D_n w_S=
 \begin{cases}
 0,& |S|=0\text{ or }1,\\
 4(1-2^{1-|S|})w_S,& |S|\geq2.
 \end{cases}
\end{equation}
All nonzero eigenvalues in this formula are at least $2$.
These are eigenvalues of the endomorphism $D_n^\dagger D_n$ of $L^2(\Prob_n)$.
The operator $D_n$ itself has a different target space, so this is not an eigenvalue assertion for $D_n$.
\end{corollary}
\begin{proof}
Apply \eqref{eq:normal} to the Walsh basis using \eqref{eq:noise-spectrum}. Since $\norm{D_nf}_2^2=\pair{D_n^\dagger D_nf}{f}$, the kernel is precisely the sum of Walsh levels zero and one.
\end{proof}

\subsubsection{The \texorpdfstring{$L^{5/4}$}{L(5/4)} estimate}

We use the following classical scalar form of Bernoulli hypercontractivity.
Bonami's general product inequality \cite[Chapter~III, Theorem~3]{Bon70} gives the condition $(q-1)\Theta^2\leq p-1$ for the $L^p$ to $L^q$ estimate, where $1<p\leq q<\infty$.
Here we need $q=2$.
The two-point inequality and its tensorization also appear in Beckner \cite[Section~II]{Bec75}; see \cite[Chapters~9--10]{ODonnell} for a systematic account.
We include the proof in the precise range used below.

\begin{theorem}[Bernoulli hypercontractivity]\label{thm:HC}
For $1<p\leq2$ and $0\leq\Theta\leq\sqrt{p-1}$,
\begin{equation}\label{eq:HC}
 \norm{T_\Theta u}_{L^2(\Prob_n)}\leq\norm{u}_{L^p(\Prob_n)}
 \qquad(u:\Delta_n\to\mathbb R).
\end{equation}
The constant is $1$ for every $n$. In particular,
\begin{equation}\label{eq:HC-special}
 \norm{T_{1/2}u}_2\leq\norm{u}_{5/4},
 \qquad
 \norm{T_{1/\sqrt2}u}_2\leq\norm{u}_{3/2}.
\end{equation}
\end{theorem}

\begin{lemma}[Two-point inequality]\label{lem:two-point}
For $1<p\leq2$ and $a,b\in\mathbb R$,
\begin{equation}\label{eq:two-point}
 a^2+(p-1)b^2
 \leq\left(\frac{|a+b|^p+|a-b|^p}{2}\right)^{2/p}.
\end{equation}
Consequently the one-coordinate noise operator satisfies
\[
 \norm{T_{\sqrt{p-1}}u}_{L^2(\{-1,1\})}
 \leq\norm{u}_{L^p(\{-1,1\})}
\]
for $u:\{-1,1\}\to\mathbb R$, with normalized counting measure.
\end{lemma}
\begin{proof}
The case $p=2$ is an identity, with noise operator $T_1=I$. Suppose $1<p<2$. If $a=0$, then \eqref{eq:two-point} follows from $p-1\leq1$. Assume $a\ne0$ and set
\[
 A(t)\coloneqq\frac{|a+t|^p+|a-t|^p}{2},
 \qquad F(t)\coloneqq A(t)^{2/p}.
\]
The function $A$ is strictly positive. Away from $t=\pm a$, differentiation gives
\[
 A''(t)=p(p-1)\frac{|a+t|^{p-2}+|a-t|^{p-2}}2
\]
and
\begin{align*}
 F''(t)
 &=\frac2p\left(\frac2p-1\right)A(t)^{2/p-2}(A'(t))^2
   +\frac2p A(t)^{2/p-1}A''(t)\\
 &\geq2(p-1)A(t)^{(2-p)/p}
       \frac{|a+t|^{p-2}+|a-t|^{p-2}}2.
\end{align*}
The function $s\mapsto s^{(p-2)/p}$ is convex on $(0,\infty)$, so its defining midpoint convexity inequality shows that the final product, without the factor $2(p-1)$, is at least $1$. Hence $F''(t)\geq2(p-1)$ wherever the displayed derivatives are finite.

The derivative $F'$ is locally absolutely continuous. Indeed, the possible singularities in $A''$ have order $|t\mp a|^{p-2}$, which is locally integrable because $p>1$, and $A$ is bounded away from zero on compact intervals. Since $F(0)=a^2$, $F'(0)=0$, and $F'$ is locally absolutely continuous, for $t\geq0$ we have
\[
 F(t)=F(0)+tF'(0)+\int_0^t(t-s)F''(s)\,ds
 \geq a^2+(p-1)t^2.
\]
The second derivative is used only almost everywhere, so this identity also crosses the two possible singular points.
Evenness of $F$ and $t=|b|$ prove \eqref{eq:two-point}.

For a real function $u$ on $\{-1,1\}$ write $u(\epsilon)=a+b\epsilon$. Its noisy version is $a+\sqrt{p-1}\,b\epsilon$, whose squared $L^2$ norm is the left-hand side of \eqref{eq:two-point}.
\end{proof}

\begin{proof}[Proof of Theorem~\ref{thm:HC}]
Fix $1<p\leq2$ and first take $\Theta=\sqrt{p-1}$. The one-coordinate assertion is Lemma~\ref{lem:two-point}. Suppose the assertion is known on $\Delta_{n-1}$. Write the variables as $(u,t)\in\Delta_{n-1}\times\{-1,1\}$, and define
\begin{align*}
 (Vf)(u,t)&\coloneqq\frac{1+\Theta}{2}f(u,t)
                  +\frac{1-\Theta}{2}f(u,-t),\\
 (Sf)(u,t)&\coloneqq\int_{\Delta_{n-1}}f(u\odot\epsilon,t)
                         \,d\nu_{n-1,\Theta}(\epsilon).
\end{align*}
Thus $V$ averages function values in the last coordinate while $S$ does so in the first $n-1$ coordinates.
The product identity $\nu_{n,\Theta}=\nu_{n-1,\Theta}\otimes\nu_{1,\Theta}$ and Fubini give $SV=VS=T_\Theta$ on $\Delta_n$.

Apply the induction hypothesis for each fixed $t$, then the mixed-norm inequality in Lemma~\ref{lem:mixed-interpolation} (since $p\leq2$), and finally the one-coordinate inequality for each fixed $u$:
\begin{align*}
 \norm{SVf}_{L^2_t(L^2_u)}
 &\leq\norm{Vf}_{L^2_t(L^p_u)}\\
 &\leq\norm{Vf}_{L^p_u(L^2_t)}\\
 &\leq\norm{f}_{L^p_u(L^p_t)}
 =\norm{f}_{L^p(\Prob_n)}.
\end{align*}
All norms here use probability measures. This proves the assertion at $\Theta=\sqrt{p-1}$ for every $n$. If $0\leq\Theta\leq\sqrt{p-1}$, use
\[
 T_\Theta=T_{\Theta/\sqrt{p-1}}T_{\sqrt{p-1}}
\]
and the $L^2$ contractivity of the first factor. This proves \eqref{eq:HC}.
\end{proof}

Let $\Pi_n$ be the orthogonal $L^2(\Prob_n)$ projection onto $\A_n$:
\begin{equation}\label{eq:Q}
 \Pi_nf=(\E f)\one+\sum_{j=1}^n(\E fr_j)r_j.
\end{equation}
All functions on a finite cube belong to $L^2$, so this definition requires no extra integrability assumption.

\begin{proposition}\label{prop:Lp}
For every $n$ and $f:\Delta_n\to\mathbb R$,
\begin{equation}\label{eq:Lp-estimate}
 \norm{f-\Pi_nf}_{5/4}\leq5\norm{D_nf}_{5/4}.
\end{equation}
Consequently,
\begin{equation}\label{eq:Lp-distance}
 \dist_{L^{5/4}(\Prob_n)}(f,\A_n)\leq5\norm{D_nf}_{5/4}.
\end{equation}
\end{proposition}
\begin{proof}
Set $h=f-\Pi_nf$ and $g=D_n^\dagger D_nf$. The function $h$ has only Walsh levels of degree at least two. Since $D_n\Pi_nf=0$ and $P_0h=0$, formula \eqref{eq:normal} gives
\begin{equation}\label{eq:g-h}
 g=4(I-2T)h.
\end{equation}
On each Walsh level of degree $k\geq2$, the multiplier $4(1-2^{1-k})$ is at least $2$. The operators in \eqref{eq:g-h} commute with $T$. Parseval's identity therefore gives
\begin{equation}\label{eq:smoothed-spectral}
 \norm{Th}_2\leq\tfrac12\norm{Tg}_2
 \leq\tfrac12\norm{g}_{5/4},
\end{equation}
where the last step is \eqref{eq:HC-special}. Solving \eqref{eq:g-h} for $h$ gives $h=g/4+2Th$. Since the underlying measure is a probability measure,
\begin{align*}
 \norm{h}_{5/4}
 &\leq\tfrac14\norm{g}_{5/4}+2\norm{Th}_{5/4}\\
 &\leq\tfrac14\norm{g}_{5/4}+2\norm{Th}_2
 \leq\tfrac54\norm{g}_{5/4}
 \leq5\norm{D_nf}_{5/4},
\end{align*}
using \eqref{eq:adjoint-bound} in the last step. The distance estimate follows by using $\Pi_nf\in\A_n$ as an approximant.
\end{proof}

\begin{remark}\label{rem:no-L1-Q}
No uniform $L^1$ bound for $\Pi_n$ is assumed or deduced. The estimate just proved applies in $L^{5/4}$. The $L^1$ argument below works with the quotient distance, not with $\norm{f-\Pi_nf}_1$.
\end{remark}

\subsection{A tail estimate from correlated small sets}\label{sec:tails}

The passage to $L^1$ requires control of large values, not merely a spectral gap in $L^2$.
We use the classical small-set consequence of hypercontractivity; see \cite[Section~9.5]{ODonnell}.
For the present coordinate-exchange coupling, the same estimate is used in conditional-probability form in \cite[p.~33:4]{BP21}.
We retain its short proof to fix the measure, the parameter, and the exponent.

\begin{lemma}\label{lem:small-sets}
For every $B\subseteq\Delta_n$ and all distinct $i,j\in\{0,1,2,3\}$,
\begin{equation}\label{eq:small-sets}
 Q_n(U_i\in B,\ U_j\in B)\leq\Prob_n(B)^{4/3}.
\end{equation}
\end{lemma}
\begin{proof}
For the independent pairs in Lemma~\ref{lem:coupling}, the left-hand side is $\Prob_n(B)^2$, which is at most $\Prob_n(B)^{4/3}$. For any of the other pairs it equals
\begin{align*}
 \pair{\one_B}{T_{1/2}\one_B}
 &=\norm{T_{1/\sqrt2}\one_B}_2^2\\
 &\leq\norm{\one_B}_{3/2}^2=\Prob_n(B)^{4/3}.
\end{align*}
Here $T_{1/2}=T_{1/\sqrt2}T_{1/\sqrt2}$ and self-adjointness give the equality.
For the inequality, use $p=3/2$ and $\Theta=1/\sqrt2=\sqrt{p-1}$ in Theorem~\ref{thm:HC}.
Finally, $\norm{\one_B}_{3/2}=\Prob_n(B)^{2/3}$, giving the exponent $4/3$ after squaring.
In general, the same argument gives
\[
 \int\one_BT_\Theta\one_B\,d\Prob_n
 \leq\Prob_n(B)^{2/(1+\Theta)}\quad(0\leq\Theta\leq1):
\]
for $\Theta>0$ take $p=1+\Theta$ and use $T_\Theta=T_{\sqrt\Theta}^2$; for $\Theta=0$, use $T_0=P_0$.
\end{proof}

\begin{lemma}[Tail estimate]\label{lem:tail}
Let $f:\Delta_n\to\mathbb R$ satisfy $\norm{f}_1=1$, and put $\varepsilon=\norm{D_nf}_1$. Then for every $M>0$,
\begin{equation}\label{eq:tail}
 \int_{\{|f|>M\}}|f|\dd\Prob_n
 \leq2\varepsilon+\gamma M^{-1/3},
 \qquad \gamma=6^{7/3}.
\end{equation}
The estimate is independent of $n$.
\end{lemma}
\begin{proof}
Define
\[
 a(t)\coloneqq\Prob_n\{|f|>t\}\quad(t>0),
 \qquad Z_i\coloneqq|f(U_i)|\quad(0\leq i\leq3).
\]
The distribution identities in Lemma~\ref{lem:distribution-tools} give
\begin{equation}\label{eq:distribution}
 a(t)\leq t^{-1},\qquad \int_0^\infty a(t)\dd t=1.
\end{equation}
Each $Z_i$ has this same tail distribution. We do not assume that these four random variables are independent.

On the event
\[
 Z_0>M,\qquad Z_j\leq Z_0/6\quad(j=1,2,3),
\]
the reverse triangle inequality gives
\[
 |D_nf|\geq Z_0-Z_1-Z_2-Z_3\geq Z_0/2.
\]
Splitting off this event and using the union bound on its complement yields
\begin{equation}\label{eq:tail-split}
 \E_{Q_n}[Z_0\one_{\{Z_0>M\}}]
 \leq2\varepsilon+
 \sum_{j=1}^3\E_{Q_n}[Z_0\one_{\{Z_0>M,\ Z_j>Z_0/6\}}].
\end{equation}
For $j\in\{1,2,3\}$ set $V_j\coloneqq\min(Z_0,6Z_j)$. On the event appearing in the $j$th summand, $V_j=Z_0>M$. Hence
\begin{equation}\label{eq:min-domination}
 \E_{Q_n}[Z_0\one_{\{Z_0>M,\ Z_j>Z_0/6\}}]
 \leq\E_{Q_n}[V_j\one_{\{V_j>M\}}].
\end{equation}
For every $t>0$, Lemma~\ref{lem:small-sets}, applied to the \emph{same} level set at threshold $t/6$, gives
\begin{align}
 Q_n(V_j>t)
 &=Q_n(Z_0>t,\ Z_j>t/6)\notag\\
 &\leq Q_n(Z_0>t/6,\ Z_j>t/6)
 \leq a(t/6)^{4/3}.\label{eq:min-tail}
\end{align}
The truncated layer-cake identity, followed by \eqref{eq:min-tail}, shows that
\begin{equation}\label{eq:layer-min}
 \E_{Q_n}[V_j\one_{\{V_j>M\}}]
 \leq M a(M/6)^{4/3}+\int_M^\infty a(t/6)^{4/3}\dd t.
\end{equation}
The first term is at most $6^{4/3}M^{-1/3}$ by \eqref{eq:distribution}. For $t\geq M$, the same estimate gives
\[
 a(t/6)^{4/3}\leq(6/M)^{1/3}a(t/6).
\]
Therefore the integral in \eqref{eq:layer-min} is at most
\[
 (6/M)^{1/3}\int_M^\infty a(t/6)\dd t
 =6(6/M)^{1/3}\int_{M/6}^\infty a(s)\dd s
 \leq6^{4/3}M^{-1/3}.
\]
Each of the three summands in \eqref{eq:tail-split} is consequently bounded by $2\cdot6^{4/3}M^{-1/3}$. Summation proves \eqref{eq:tail}, because $3\cdot2\cdot6^{4/3}=6^{7/3}$.
\end{proof}

\begin{remark}\label{rem:homogeneous-tail}
Rescaling gives the homogeneous form
\[
 \int_{\{|f|>M\norm{f}_1\}}|f|\dd\Prob_n
 \leq2\norm{D_nf}_1+\gamma M^{-1/3}\norm{f}_1
\]
for nonzero $f$. The tail estimate applies to all functions, not only to functions orthogonal to $\A_n$. For affine functions the defect vanishes, which is consistent with their uniformly controlled tails after $L^1$ normalization.
\end{remark}

\subsection{The dimension-free \texorpdfstring{$L^1$}{L1} quotient estimate}\label{sec:endpoint}

Define the following absolute constants:
\begin{equation}\label{eq:constants}
 \begin{gathered}
 \gamma=6^{7/3},\qquad A_0=2+\gamma,\\
 B_0=A_0+5\cdot4^{1/5}(1+4A_0)^{4/5},
 \qquad C=B_0^5.
 \end{gathered}
\end{equation}
No optimization is intended.

\begin{theorem}\label{thm:finite-endpoint}
For every $n\geq1$ and $f:\Delta_n\to\mathbb R$,
\begin{equation}\label{eq:finite-endpoint}
 \dist_{L^1(\Prob_n)}(f,\A_n)\leq C\norm{D_nf}_1.
\end{equation}
Together with \eqref{eq:D-bounded}, this gives
\[
 C^{-1}\dist_{L^1(\Prob_n)}(f,\A_n)
 \leq\norm{D_nf}_1
 \leq4\dist_{L^1(\Prob_n)}(f,\A_n).
\]
\end{theorem}
\begin{proof}
Write $d=\dist_{L^1(\Prob_n)}(f,\A_n)$. If $d=0$, the assertion is immediate. Otherwise, a best approximant $a_*\in\A_n$ exists. Indeed, a minimizing sequence in $\A_n$ is bounded by the triangle inequality; finite dimensionality gives a convergent subsequence. Replacing $f$ by $(f-a_*)/d$, and using $D_na_*=0$, reduces the proof to the normalization
\begin{equation}\label{eq:normalization}
 \norm{f}_1=\dist_{L^1(\Prob_n)}(f,\A_n)=1.
\end{equation}
Set $\varepsilon=\norm{D_nf}_1$. By Corollary~\ref{cor:finite-kernel}, $\varepsilon>0$. If $\varepsilon\geq1$, then $1\leq C\varepsilon$. It remains to consider $0<\varepsilon\leq1$.

Choose
\begin{equation}\label{eq:cutoff}
 M=\varepsilon^{-3},\qquad
 b=f\one_{\{|f|\leq M\}},\qquad
 \delta=\norm{f-b}_1.
\end{equation}
Lemma~\ref{lem:tail} gives
\begin{equation}\label{eq:delta}
 \delta\leq2\varepsilon+\gamma M^{-1/3}
 =A_0\varepsilon.
\end{equation}
By boundedness of $D_n$ and the pointwise bound on $b$,
\begin{equation}\label{eq:Db-endpoints}
 \norm{D_nb}_1\leq\varepsilon+4\delta
 \leq(1+4A_0)\varepsilon,
 \qquad \norm{D_nb}_\infty\leq4M.
\end{equation}
For any bounded function $v$ on a probability space,
\[
 \norm{v}_{5/4}\leq\norm{v}_\infty^{1/5}\norm{v}_1^{4/5};
\]
this follows directly by integrating $|v|^{5/4}\leq\norm{v}_\infty^{1/4}|v|$. Thus \eqref{eq:Db-endpoints} and \eqref{eq:cutoff} imply
\begin{align}
 \norm{D_nb}_{5/4}
 &\leq (4\varepsilon^{-3})^{1/5}
       ((1+4A_0)\varepsilon)^{4/5}\notag\\
 &=4^{1/5}(1+4A_0)^{4/5}\varepsilon^{1/5}.
 \label{eq:Db-Lp}
\end{align}
Using \eqref{eq:normalization}, the triangle inequality for distance, and Proposition~\ref{prop:Lp}, we obtain
\begin{align}
 1
 &\leq\norm{f-b}_1+\dist_{L^1(\Prob_n)}(b,\A_n)\notag\\
 &\leq\delta+\dist_{L^{5/4}(\Prob_n)}(b,\A_n)\notag\\
 &\leq A_0\varepsilon
      +5\cdot4^{1/5}(1+4A_0)^{4/5}\varepsilon^{1/5}
 \leq B_0\varepsilon^{1/5}.
 \label{eq:absorption}
\end{align}
The second line uses $\norm{u}_1\leq\norm{u}_{5/4}$ for each candidate residual $u=b-a$. The final step uses $\varepsilon\leq1$. It follows that $\varepsilon\geq B_0^{-5}=C^{-1}$. Undoing the normalization gives \eqref{eq:finite-endpoint}.

For the upper estimate, apply $\norm{D_n(f-a)}_1\leq4\norm{f-a}_1$ to each $a\in\A_n$ and take the infimum.
\end{proof}

\subsection{Passage to the countable cube and the embedding}\label{sec:infinite}

We now work on $(\Delta,\Prob)$ from the introduction. Let $\F_n$ be the sigma-algebra generated by $r_1,\ldots,r_n$, and let
\[
 E_nf=\E[f\mid\F_n].
\]
We identify $\F_n$-measurable functions with functions on $\Delta_n$.

\begin{lemma}\label{lem:conditional}
The maps $E_n$ are contractions on $L^1(\Prob)$, and $E_nf\to f$ in $L^1(\Prob)$ for each $f\in L^1(\Prob)$. Moreover,
\begin{equation}\label{eq:conditional-A}
 E_nR=\spanop\{r_1,\ldots,r_n\},\qquad E_n\A=\A_n.
\end{equation}
For every $\F_n$-measurable $u$,
\begin{equation}\label{eq:distance-reflection}
 \dist_{L^1(\Prob)}(u,\A)
 =\dist_{L^1(\Prob_n)}(u,\A_n).
\end{equation}
\end{lemma}
\begin{proof}
Contractivity is Jensen's inequality. Finite-coordinate cylinder sets form an algebra generating the product $\sigma$-algebra. The monotone-class theorem and dominated convergence therefore imply that cylinder functions are dense in $L^1(\Prob)$. Given a cylinder function $v$ depending on the first $m$ coordinates, $E_nv=v$ for $n\geq m$; consequently
\[
 \norm{E_nf-f}_1\leq2\norm{f-v}_1\qquad(n\geq m).
\]
Density proves convergence.

Independence gives $E_nr_j=r_j$ if $j\leq n$, and $E_nr_j=0$ if $j>n$. By continuity and finite-dimensional closedness, $E_nR$ is contained in the indicated span. The reverse inclusion follows because $E_n$ fixes that span. The assertion about $\A$ follows as well.

Since $\A_n\subseteq\A$, the left side of \eqref{eq:distance-reflection} is at most the right side. Conversely, for any $a\in\A$, the element $E_na$ belongs to $\A_n$, and
\[
 \norm{u-E_na}_1=\norm{E_n(u-a)}_1\leq\norm{u-a}_1.
\]
Taking infima proves the reverse inequality. The $L^1$ norm of a cylinder function agrees exactly with its norm on the finite cube.
\end{proof}

\begin{lemma}\label{lem:infinite-D}
Formula \eqref{eq:D-infinite} defines a bounded linear operator
\[
 D:L^1(\Prob)\longrightarrow L^1(\Prob^{\otimes3})
\]
of norm at most $4$. It vanishes on $\A$. If $u$ depends on the first $n$ coordinates, then $Du$ is the corresponding cylinder realization of $D_nu$ and
\begin{equation}\label{eq:cylinder-D}
 \norm{Du}_{L^1(\Prob^{\otimes3})}
 =\norm{D_nu}_{L^1(\Prob_n^{\otimes3})}.
\end{equation}
\end{lemma}
\begin{proof}
The coordinate maps in \eqref{eq:swaps} are measurable. All four points $x,y,z,w$ have distribution $\Prob$: this follows from their finite-dimensional distributions, or directly from the coordinatewise construction. Pullback along each of these four maps therefore preserves the $L^1$ norm. It also respects equivalence classes, since the inverse image of a $\Prob$-null set is $\Prob^{\otimes3}$-null. Formula \eqref{eq:D-infinite} therefore defines $D$ with norm at most $4$.

The identities $D\one=0$ and $Dr_j=0$, followed by continuity and \eqref{eq:spaces}, show that $D$ vanishes on $\A$. For cylinder functions the defining formula uses only the corresponding finite coordinates, proving \eqref{eq:cylinder-D}.
\end{proof}

\begin{proposition}\label{prop:infinite-distance}
The estimate \eqref{eq:main-distance} holds, and $\ker D=\A$. In particular,
\[
 \widetilde D:L^1(\Prob)/\A\longrightarrow L^1(\Prob^{\otimes3}),
 \qquad \widetilde D(f+\A)=Df,
\]
is an isomorphism onto a closed subspace, with $\norm{\widetilde D}\leq4$ and $\norm{\widetilde D^{-1}}\leq C$ on its range.
\end{proposition}
\begin{proof}
Apply Theorem~\ref{thm:finite-endpoint} to $f_n=E_nf$. By Lemmas~\ref{lem:conditional} and~\ref{lem:infinite-D},
\[
 \dist_{L^1(\Prob)}(f_n,\A)\leq C\norm{Df_n}_1.
\]
As $n\to\infty$, we have $f_n\to f$ and $Df_n\to Df$ in their respective $L^1$ spaces. Distance to a fixed subset is $1$-Lipschitz, so the lower bound in \eqref{eq:main-distance} follows. The upper bound follows by subtracting any $a\in\A$, since $Da=0$.

Closedness of $\A$ and the lower bound identify the kernel. The induced map has the stated two-sided bounds. A bounded-below operator from a Banach space has closed range, which proves the remaining assertion.
\end{proof}

\begin{proposition}\label{prop:J}
The map $J$ in \eqref{eq:main-J} is well defined and satisfies \eqref{eq:main-J-bounds}. Its range is closed.
\end{proposition}
\begin{proof}
Both the mean and $D$ vanish on $R$, so $J$ is well defined and linear. For every $r\in R$,
\[
 |\E f|\leq\norm{f-r}_1,
 \qquad \norm{Df}_1=\norm{D(f-r)}_1\leq4\norm{f-r}_1.
\]
Taking the infimum gives $\norm{J(f+R)}\leq5\norm{f+R}_{L^1/R}$.

For $c\in\mathbb R$ and $r\in R$, write $u=f-c\one-r$. Since $\E r=0$,
\[
 |c|=|\E f-\E u|\leq|\E f|+\norm{u}_1.
\]
Hence
\[
 \dist(f,R)\leq\norm{f-r}_1\leq|\E f|+2\norm{u}_1.
\]
Infimizing over $c\one+r\in\A$ and using Proposition~\ref{prop:infinite-distance}, we obtain
\begin{equation}\label{eq:mean-recovery}
 \dist(f,R)\leq|\E f|+2\dist(f,\A)
 \leq|\E f|+2C\norm{Df}_1
 \leq2C\norm{J(f+R)}.
\end{equation}
Here $C\geq1$. This is the lower estimate in \eqref{eq:main-J-bounds}, and it also implies closedness of the range.
\end{proof}

\subsubsection{Realization in the usual \texorpdfstring{$L^1(0,1)$}{L1(0,1)}}

For almost every $t\in[0,1)$, write its unique nonterminating binary expansion as
\[
 t=\sum_{j\geq1}\beta_j(t)2^{-j},\qquad \beta_j(t)\in\{0,1\}.
\]
Any convention at dyadic rationals gives the same $L^1$ identifications. Under Lebesgue measure, these digits are independent fair bits, and the map
\[
 t\longmapsto\bigl((-1)^{\beta_j(t)}\bigr)_{j\geq1}
\]
is a probability-space isomorphism modulo null sets. Its pullback is an isometric isomorphism from $L^1(\Prob)$ onto $L^1(0,1)$, taking $r_j$ to the usual Rademacher function $(-1)^{\lfloor2^jt\rfloor}$ almost everywhere.

Interlacing the binary digits in three subsequences likewise gives a probability-space isomorphism $\chi:[0,1)\to\Delta^3$ modulo null sets. Explicitly, its three sequences use the digits $\beta_{3j-2}$, $\beta_{3j-1}$, and $\beta_{3j}$, respectively. Therefore
\[
 Uh=h\circ\chi
\]
is an isometric isomorphism $L^1(\Prob^{\otimes3})\to L^1(0,1)$. Define
\begin{equation}\label{eq:target-realization}
 \mathcal I(c,h)(t)=
 \begin{cases}
 2c,&0<t<1/2,\\
 2(Uh)(2t-1),&1/2<t<1.
 \end{cases}
\end{equation}
A change of variables gives
\[
 \norm{\mathcal I(c,h)}_{L^1(0,1)}=|c|+\norm{h}_{L^1(\Prob^{\otimes3})}.
\]
Thus $\mathcal I$ is a linear isometry from $\mathbb R\oplus_1L^1(\Prob^{\otimes3})$ into $L^1(0,1)$. Composing it with $J$ proves the asserted embedding of $L^1/R$. The map $U\widetilde D$ treats $L^1/\A$.

Finally, identify $[0,1)$ with the torus $\mathbb T=\mathbb R/\mathbb Z$ endowed with normalized Haar measure. This identification preserves the Rademacher system modulo null sets. Thus the same conclusions hold for the usual Rademacher system on the torus. This completes the proof of Theorem~\ref{thm:main}.

\newpage
\subsection{Closed-range formulation and normalization}\label{sec:consequences}

Tauberian operators on $L^1$ with infinite-dimensional kernels were constructed in \cite{JNST}. The construction here has closed range and prescribes the classical Rademacher span as its kernel.

The following is the classical identification of the Rademacher span with $\ell_2$, obtained from Khintchine's inequalities; see \cite[Theorem~6.2.2 and Remark~6.2.3(a)]{AK16}.
For the reader's convenience, we include a proof with the constants used below.

\begin{lemma}\label{lem:R-Hilbert}
The coefficient map $(a_j)\mapsto\sum_j a_jr_j$ extends to an isomorphism from $\ell_2(\mathbb R)$ onto $R$. More precisely, for every finitely supported scalar sequence,
\begin{equation}\label{eq:Khintchine-rough}
 \frac1{\sqrt3}\Bigl(\sum_j|a_j|^2\Bigr)^{1/2}
 \leq\left\lVert\sum_j a_jr_j\right\rVert_1
 \leq\Bigl(\sum_j|a_j|^2\Bigr)^{1/2}.
\end{equation}
In particular, $R$ is reflexive and infinite dimensional.
\end{lemma}
\begin{proof}
Let $s=\sum_j a_jr_j$ and $v^2=\sum_j|a_j|^2$. Independence gives $\norm{s}_2=v$. Expanding the fourth moment gives
\[
 \E s^4=3v^4-2\sum_j a_j^4\leq3v^4.
\]
H\"older's inequality gives $\norm{s}_2\leq\norm{s}_1^{1/3}\norm{s}_4^{2/3}$. If $v>0$, this implies $\norm{s}_1\geq v^3/\norm{s}_4^2\geq v/\sqrt3$. The upper bound is $\norm{s}_1\leq\norm{s}_2$. Completion yields an isomorphism onto the closed span of the $r_j$.
\end{proof}

Recall that a bounded operator $T:X\to Y$ is \emph{Tauberian} if
\[
 (T^{**})^{-1}(J_Y Y)=J_X X,
\]
where $J_X,J_Y$ denote the canonical embeddings into the biduals. This terminology and its connection with quotients by reflexive subspaces are discussed in \cite{GM97}. The following standard closed-range observation is recorded with its proof.

\begin{lemma}\label{lem:Tauberian}
If $T:X\to Y$ has closed range and reflexive kernel, then it is Tauberian.
\end{lemma}
\begin{proof}
Write $Z=\ker T$ and $F=\ran T$. Closedness of $F$ gives a constant $c$ such that $\dist(x,Z)\leq c\norm{Tx}$ for $x\in X$. For each $x^*\in Z^\perp$, the rule $Tx\mapsto x^*(x)$ defines a bounded functional on $F$, and the Hahn--Banach theorem extends it to $Y$. It follows that $\ran T^*=Z^\perp$ and hence
\[
 \ker T^{**}=Z^{\perp\perp}.
\]
Under the natural identification of $Z^{**}$ with $Z^{\perp\perp}\subset X^{**}$, reflexivity gives $Z^{\perp\perp}=J_X Z$.

Suppose that $T^{**}x^{**}=J_Yy$ with $y\in Y$. Every functional annihilating $F$ annihilates $y$, so $y\in F$ by closedness of $F$ and Hahn--Banach separation. Choose $x_0\in X$ with $Tx_0=y$. Then $x^{**}-J_Xx_0\in\ker T^{**}=J_XZ$, so $x^{**}\in J_XX$. The converse inclusion is immediate from $T^{**}J_X=J_YT$.
\end{proof}

\begin{corollary}\label{cor:closed-range}
There exists a bounded operator $\mathcal T:L^1(0,1)\to L^1(0,1)$ with closed range and kernel exactly the usual Rademacher span $R$. It satisfies
\[
 \frac1{2C}\dist(f,R)\leq\norm{\mathcal T f}_1
 \leq5\dist(f,R).
\]
This operator is Tauberian and is not upper semi-Fredholm.
\end{corollary}
\begin{proof}
On the Bernoulli model put
\[
 \mathcal T f=\mathcal I(\E f,Df),
\]
and use the source identification with $L^1(0,1)$ from Subsection~\ref{sec:infinite}. Proposition~\ref{prop:J} proves the kernel, range, and norm assertions. Lemmas~\ref{lem:R-Hilbert} and~\ref{lem:Tauberian} imply that $\mathcal T$ is Tauberian. An upper semi-Fredholm operator has closed range and finite-dimensional kernel; the latter condition fails here.
\end{proof}

\subsubsection{The role of constants and the size of the estimate}

\begin{remark}\label{rem:convention}
The operator $D$ annihilates constants as well as $R$. Therefore $f+R\mapsto Df$ alone would not be injective. Formula \eqref{eq:mean-recovery} is the precise reason that adding the single scalar coordinate $\E f$ suffices. Under the alternative convention in which the constant function belongs to the Rademacher space, the relevant space is $\A$, and $\widetilde D$ itself is the embedding.
\end{remark}

\begin{example}\label{ex:two-dim}
For the two-dimensional cube take $f=r_1r_2$. Since $|f|=1$ and $\E(fa)=0$ for $a\in\A_2$,
\[
 \norm{f-a}_1\geq|\E((f-a)f)|=1,
 \qquad \dist_{L^1}(f,\A_2)=1.
\]
If the two swap signs agree, then $D_2f=0$. If they disagree, direct calculation gives
\[
 D_2f=(x_1-y_1)(x_2-y_2).
\]
Thus $|D_2f|=4$ precisely when $x_1\ne y_1$, $x_2\ne y_2$, and $\sigma_1\ne\sigma_2$, an event of probability $1/8$. Consequently $\norm{D_2f}_1=1/2$. Any constant in \eqref{eq:finite-endpoint} must therefore be at least $2$. The much larger value in \eqref{eq:constants} is only an explicit existence bound; no sharpness claim is made.
\end{example}

\section{Weak essential norms and incompressibility}\label{sec:algebra}
We retain the classical pair $(E,R)$ and quotient $X=E/R$ from Section~\ref{sec:embedding}.
For the extension of the argument, let $E_\mu\coloneqq L^1(\mu)$ be an arbitrary $L^1$-space and suppose that there are bounded linear maps
\begin{equation}\label{c:eq:ambient}
 \iota:E\to E_\mu,\qquad P:E_\mu\to E,\qquad
 \norm{\iota f}=\norm{f},\quad\norm{P}\leq1,\quad P\iota=\Id_{E}.
\end{equation}
Set $R_\mu\coloneqq\iota R$, $X_\mu\coloneqq E_\mu/R_\mu$, and let $q_\mu:E_\mu\to X_\mu$ be the metric quotient.
The classical case is $E_\mu=E$, $\iota=P=\Id_E$.
An $AL$-space is a Banach lattice whose norm is additive on the positive cone; every $L^1$-space and every finite $\ell_1$-sum of such spaces is an $AL$-space.

We write $\Comp(Y,Z)$ for the compact operators and
\[
 \Ap(Y,Z)=\overline{\{\text{finite-rank operators }Y\to Z\}}^{\norm{\cdot}}
 \subseteq\Comp(Y,Z)
\]
for the approximable operators.
No approximation property of $X_\mu$ will be assumed.
All references to the ambient condition below mean \eqref{c:eq:ambient}.

\subsection{Residuum operators, reflexive quotients, and liftings}

\subsubsection{The residuum functor}
Composition and identities are preserved:
\[
 (BA)^{\co}=B^{\co}A^{\co},\qquad
 (\Id_Y)^{\co}=\Id_{Y^{\co}},\qquad
 \norm{A^{\co}}\leq\norm{A}.
\]
Gantmacher's criterion gives $A\in\W(Y,Z)$ if and only if $A^{**}(Y^{**})\subseteq J_ZZ$, equivalently $A^{\co}=0$. One way to see the criterion is to combine Goldstine's theorem with weak-star compactness of $B_{Y^{**}}$: the weak-star closure of $J_ZA(B_Y)$ is $A^{**}(B_{Y^{**}})$, and on $J_ZZ$ the induced weak-star topology is the weak topology of $Z$. Therefore
\begin{equation}\label{c:eq:coineq}
 \norm{A^{\co}}\leq\norm{A}_w.
\end{equation}
The representation $\Qw(Y)\to\B(Y^{\co})$ is consequently contractive and injective.

\begin{lemma}[The factor $2$ for an embedding]\label{c:lem:coembedding}
If $D:Y\to Z$ satisfies $\norm{Dy}\geq m\norm{y}$ for some $m>0$, then
\begin{equation}\label{c:eq:coembedding}
 \norm{D^{\co}\xi}\geq\tfrac m2\norm{\xi}\qquad(\xi\in Y^{\co}).
\end{equation}
\end{lemma}
\begin{proof}
Put $M=D(Y)$, a closed subspace of $Z$. If $z^{**}\in M^{\perp\perp}$ and $z\in Z$, then
\[
 \dist(z,M)=\sup_{\substack{z^*\in M^\perp\\\norm{z^*}\leq1}}\abs{z^*(z)}
 \leq\norm{z-z^{**}}.
\]
Choosing $m_0\in M$ arbitrarily close to this distance gives
\[
 \dist(z^{**},M)\leq2\dist(z^{**},Z).
\]
The inverse $D^{-1}:M\to Y$ has norm at most $m^{-1}$, and its bidual is defined on $M^{**}=M^{\perp\perp}$. Apply it to $D^{**}y^{**}-m_0$ and take infima. This yields
\[
 \dist(y^{**},Y)\leq m^{-1}\dist(D^{**}y^{**},M)
 \leq2m^{-1}\dist(D^{**}y^{**},Z).
\]
This is \eqref{c:eq:coembedding}. The estimate is the one used in \cite[Proposition~1.4(i)]{GST}.
\end{proof}

\begin{lemma}[Quotient by a reflexive subspace]\label{c:lem:refquot}
Let $Z$ be a closed reflexive subspace of a Banach space $Y$, and let $q_Y:Y\to Y/Z$ be the metric quotient. Then
\begin{equation}\label{c:eq:refquot}
 \ker q_Y^{**}=J_YZ,\qquad
 (q_Y^{**})^{-1}(J_{Y/Z}(Y/Z))=J_YY.
\end{equation}
Moreover $q_Y^{\co}:Y^{\co}\to(Y/Z)^{\co}$ is a surjective isometry.
\end{lemma}
\begin{proof}
The adjoint identifies $(Y/Z)^*$ isometrically with $Z^\perp$. Hahn--Banach therefore makes $q_Y^{**}$ a metric quotient with kernel $Z^{\perp\perp}$. Reflexivity gives $Z^{\perp\perp}=J_YZ$.
If $q_Y^{**}y^{**}=J_{Y/Z}x$, choose $y\in Y$ with $q_Yy=x$; then $y^{**}-J_Yy\in J_YZ$. This proves the second assertion in \eqref{c:eq:refquot}.
Finally,
\[
 \dist(q_Y^{**}y^{**},J_{Y/Z}(Y/Z))
 =\inf_{y\in Y}\inf_{z\in Z}\norm{y^{**}-J_Yy-J_Yz}
 =\dist(y^{**},J_YY).
\]
Surjectivity follows from that of $q_Y^{**}$.
\end{proof}

\begin{lemma}[Metric lifting]\label{c:lem:lifting}
Let $L$ be an $L^1$-space, let $Z$ be a closed reflexive subspace of $Y$, and let $A:L\to Y/Z$ be bounded. There is $\widehat A:L\to Y$ such that
\begin{equation}\label{c:eq:lifting}
 q_Y\widehat A=A,\qquad \norm{\widehat A}=\norm{A}.
\end{equation}
If $A$ is weakly compact, every bounded lifting of $A$ is weakly compact.
\end{lemma}
\begin{proof}
The dual $L^*$ is $1$-injective: it is an $AM$-space with an order unit, hence isometrically a $C(K)$-space by Kakutani's representation theorem, and every dual $C(K)$-space is $1$-injective by \cite[Proposition~4.3.8(i)]{AK16}.
Here $1$-injectivity means that every bounded operator from a subspace of a Banach space into $L^*$ extends to the whole space without increasing its norm.
Identify $(Y/Z)^*$ isometrically with $Z^\perp$ through $q_Y^*$.
The operator $A^*(q_Y^*)^{-1}:Z^\perp\to L^*$ therefore has an extension $B:Y^*\to L^*$ satisfying
\[
 Bq_Y^*=A^*,\qquad \norm{B}=\norm{A}.
\]
Taking adjoints and restricting to $J_LL$ gives
\[
 q_Y^{**}B^*J_L=A^{**}J_L=J_{Y/Z}A.
\]
By Lemma~\ref{c:lem:refquot}, $B^*J_L(L)\subseteq J_YY$. Thus
\[
 \widehat A\coloneqq J_Y^{-1}B^*J_L:L\longrightarrow Y
\]
satisfies $q_Y\widehat A=A$ and $\norm{\widehat A}\leq\norm{A}$.
The reverse inequality follows from $\norm{q_Y}\leq1$.
For any bounded lifting, $q_Y^{\co}\widehat A^{\co}=A^{\co}$.
If $A$ is weakly compact, then $A^{\co}=0$, and injectivity of $q_Y^{\co}$ gives $\widehat A^{\co}=0$.
Gantmacher's criterion completes the proof.
\end{proof}

\subsubsection{The quantitative \texorpdfstring{$AL$}{AL}-input}
We use the following external theorem, with the supremum of the empty set interpreted as zero:
\begin{theorem}[Acuaviva--Nasseri {\cite[Theorem~1.4]{AN}}]\label{c:thm:ALinput}
For $AL$-spaces $L_1,L_2$ and $A\in\B(L_1,L_2)$,
\begin{equation}\label{c:eq:ALinput}
 \norm{A}_w=
 \sup\left\{\frac1{\norm{U}\norm{V}}:
 U:L_2\to\ell_1,\ V:\ell_1\to L_1,\ UAV=\Id_{\ell_1}\right\}.
\end{equation}
\end{theorem}
This theorem is only applied with \emph{both} spaces $AL$-spaces, never directly to $X_\mu$.

\begin{corollary}\label{c:cor:ALco}
For an operator $A$ between $AL$-spaces,
\begin{equation}\label{c:eq:ALco}
 \norm{A}_w=\norm{A^{\co}}.
\end{equation}
\end{corollary}
\begin{proof}
For every factorization in \eqref{c:eq:ALinput}, functoriality gives
$U^{\co}A^{\co}V^{\co}=\Id_{\ell_1^{\co}}$.
The space $\ell_1^{\co}$ is nonzero, so
$1\leq\norm{U}\norm{A^{\co}}\norm{V}$.
Take the supremum and use \eqref{c:eq:coineq}.
\end{proof}

\subsection{Rademacher coordinates and a fixed \texorpdfstring{$\ell_1$}{ell1}-idempotent}

\begin{lemma}[Shrinking Rademacher coordinates]\label{c:lem:rad}
The space $R$ is reflexive.
Let $E_N=\E[\,\cdot\mid\F_N]$ and $Q_N=E_N|_{R}$.
Then
\begin{equation}\label{c:eq:shrinking}
 Q_Nr_j=\begin{cases}r_j,&j\leq N,\\0,&j>N,\end{cases}
 \qquad\norm{Q_N^*\varphi-\varphi}\longrightarrow0\quad(\varphi\in R^*).
\end{equation}
\end{lemma}
\begin{proof}
Lemma~\ref{lem:R-Hilbert} gives an isomorphism $U:\ell_2\to R$, $Ua=\sum_ja_jr_j$, with
\begin{equation}\label{c:eq:khintchine}
 3^{-1/2}\norm{a}_2\leq\norm{Ua}_1\leq\norm{a}_2.
\end{equation}
In particular $R$ is reflexive.
Lemma~\ref{lem:conditional} gives the formula for $Q_N$.
The operators $U^{-1}Q_NU$ are the ordinary coordinate projections on $\ell_2$.
Their adjoints converge strongly; conjugation by $U^*$ gives the asserted convergence on $R^*$.
\end{proof}

Under the ambient condition \eqref{c:eq:ambient}, the maps
\begin{equation}\label{c:eq:quotretraction}
 \bar\iota(qf)=q_\mu\iota f,\qquad
 \bar P(q_\mu g)=qPg
\end{equation}
are well-defined contractions with $\bar P\bar\iota=\Id_{X}$. Hence $\bar\iota$ is an isometry and $R_\mu$ is reflexive.

\begin{lemma}[A norm-one complemented copy of $\ell_1$]\label{c:lem:ellone}
There are contractions $j:\ell_1\to X_\mu$ and $p:X_\mu\to\ell_1$ with $pj=\Id_{\ell_1}$. In particular, $P_1=jp$ is an idempotent that is not weakly compact.
\end{lemma}
\begin{proof}
We first work in $X$. The variables $s_k=r_1r_{k+1}$ are independent symmetric signs. Define
\[
 A_n=\{s_1=-1,\ldots,s_{n-1}=-1,\ s_n=1\}.
\]
They are disjoint, have probability $2^{-n}$, and are invariant under $x\mapsto-x$. Thus $\int_{A_n}r_j\,d\Prob=0$ for every $n,j$. The formulas
\[
 j_{\Delta}a=q\left(\sum_{n\geq1}a_n\frac{\one_{A_n}}{\Prob(A_n)}\right),
 \qquad
 p_{\Delta}(qf)=\left(\int_{A_n}f\,d\Prob\right)_{n\geq1}
\]
define contractions, and $p_{\Delta}j_{\Delta}=\Id_{\ell_1}$. For $p_{\Delta}$, well-definedness follows by continuity from the vanishing integrals; its norm bound follows from disjointness and infimizing over representatives of $qf$.
Set $j=\bar\iota j_{\Delta}$ and $p=p_{\Delta}\bar P$. Then $pj=\Id_{\ell_1}$, so $j$ is an isometry. If $P_1$ were weakly compact, then $pP_1j=\Id_{\ell_1}$ would be weakly compact, contradicting nonreflexivity of $\ell_1$.
\end{proof}

\subsection{Exact approximable extensions of Rademacher restrictions}

\subsubsection{Lifting coefficient fields}
\begin{lemma}\label{c:lem:coefficientquotient}
The map
\[
 \mathcal C:L^\infty(\Delta)\to\ell_2,\qquad
 \mathcal Cg=\left(\int_\Delta gr_j\,d\Prob\right)_{j\geq1}
\]
is bounded and onto. Each $a\in\ell_2$ has a preimage $g$ with
\begin{equation}\label{c:eq:coefficientlift}
 \norm{g}_\infty\leq\sqrt3\norm{a}_2.
\end{equation}
\end{lemma}
\begin{proof}
Bessel gives $\norm{\mathcal Cg}_2\leq\norm{g}_2\leq\norm{g}_\infty$.
The functional $\sum_jc_jr_j\mapsto\sum_jc_ja_j$ has norm at most $\sqrt3\norm{a}_2$ on $R$ by \eqref{c:eq:khintchine}. Extend it to $E$ by Hahn--Banach and represent the extension by $g\in L^\infty$. We use the dual pairing $\int fg$.
\end{proof}

\begin{lemma}[Bochner lifting]\label{c:lem:bochner}
Let $Z,Y$ be Banach spaces, let $Q:Z\to Y$ be bounded, and suppose every $y\in Y$ admits $z\in Z$ with $Qz=y$ and $\norm{z}\leq c\norm{y}$. For a finite measure $\nu$ and $a\in L^1(\nu;Y)$, there exists $b\in L^1(\nu;Z)$ with
\begin{equation}\label{c:eq:bochner}
 Qb(s)=a(s)\text{ almost everywhere},\qquad
 \norm{b}_{L^1(\nu;Z)}\leq2c\norm{a}_{L^1(\nu;Y)}.
\end{equation}
\end{lemma}
\begin{proof}
The zero case is immediate. Choose finite-valued simple functions $t_n$ so that
$\norm{a-t_n}_{L^1}\leq\delta2^{-n}$ with $\delta=\norm{a}_{L^1}/2$.
Put $a_1=t_1$ and $a_n=t_n-t_{n-1}$ for $n\geq2$. Then
\[
 a=\sum_{n\geq1}a_n\text{ in }L^1(\nu;Y),\qquad
 \sum_{n\geq1}\norm{a_n}_{L^1}\leq\norm{a}_{L^1}+2\delta=2\norm{a}_{L^1}.
\]
Write each simple $a_n$ on disjoint measurable sets and choose a preimage in $Z$ for each of its finitely many values. This produces a simple $b_n$ with $Qb_n=a_n$ and $\norm{b_n}_{L^1}\leq c\norm{a_n}_{L^1}$.
The series $\sum_nb_n$ converges in the Bochner space $L^1(\nu;Z)$ to the required $b$. Continuity of the induced map on Bochner spaces gives $Qb=a$. No measurable linear right inverse is used.
\end{proof}

\subsubsection{The full infinite tail of a bounded kernel}
\begin{lemma}[Bounded kernels]\label{c:lem:boundedkernel}
Let $(S,\nu)$ be a finite measure space and let $h\in L^\infty(S\times\Delta,\nu\otimes\Prob)$. Define
\[
 H:E\to L^1(\nu),\qquad
 (Hf)(s)=\int_\Delta h(s,t)f(t)\,d\Prob(t).
\]
For every $\varepsilon>0$ there exists $C\in\Ap(E,L^1(\nu))$ satisfying
\begin{equation}\label{c:eq:boundedkernel}
 C|_{R}=H|_{R},\qquad \norm{C}\leq\norm{H}+\varepsilon.
\end{equation}
\end{lemma}
\begin{proof}
Set $a_j(s)=\int_\Delta h(s,t)r_j(t)\,d\Prob(t)$. For almost every $s$, Bessel gives
\[
 \sum_{j\geq1}\abs{a_j(s)}^2
 \leq\int_\Delta\abs{h(s,t)}^2\,d\Prob(t)\leq\norm{h}_\infty^2.
\]
Let $a^{>N}(s)$ be the vector with its first $N$ coordinates zero and remaining coordinates $a_j(s)$. These are strongly measurable $\ell_2$-valued functions: their finite-coordinate truncations are measurable and converge pointwise in $\ell_2$. Dominated convergence yields
\begin{equation}\label{c:eq:tailfield}
 \norm{a^{>N}}_{L^1(\nu;\ell_2)}\longrightarrow0.
\end{equation}
Apply Lemmas~\ref{c:lem:coefficientquotient} and~\ref{c:lem:bochner} to obtain
$b_N\in L^1(\nu;L^\infty(\Delta))$ with
\[
 \mathcal Cb_N(s)=a^{>N}(s),\qquad
 \norm{b_N}_{L^1(L^\infty)}\leq2\sqrt3\norm{a^{>N}}_{L^1(\ell_2)}.
\]
Define $B_N:E\to L^1(\nu)$ by
\[
 (B_Nf)(s)=\langle b_N(s),f\rangle.
\]
The scalar function is measurable because evaluation against $f$ is continuous on $L^\infty$, and
\begin{equation}\label{c:eq:tailoperator}
 \norm{B_N}\leq\norm{b_N}_{L^1(L^\infty)}\longrightarrow0.
\end{equation}
Approximating $b_N$ in Bochner norm by finite-valued simple functions approximates $B_N$ in operator norm by finite-rank operators. Thus $B_N$ is approximable; this construction does not require a jointly chosen pointwise representative of every $L^\infty$-class $b_N(s)$.
By construction,
\[
 B_Nr_j=0\ (j\leq N),\qquad B_Nr_j=Hr_j\ (j>N).
\]
Therefore $C_N=HE_N+B_N$ is approximable and agrees with $H$ on every $r_j$, hence on $R$. Since $\norm{HE_N}\leq\norm{H}$, \eqref{c:eq:tailoperator} proves \eqref{c:eq:boundedkernel}.
\end{proof}

\subsubsection{Weak compactness and kernel truncation}
We use the following consequence of Weis's weak-approximation theorem \cite{Weis}, in the standard finite-measure form stated in \cite[Theorem~4.1 and Remark~4.2]{AN}.

\begin{theorem}[Classical bounded-kernel approximation, Weis]\label{c:lem:kernelapprox}
Let the source and target carry finite Borel measures on compact metrizable spaces, with completions allowed. Every weakly compact operator between their $L^1$-spaces is an operator-norm limit of operators with bounded measurable kernels.
\end{theorem}
\begin{proof}[Derivation from the cited theorem]
Write $H=S+K$ as its singular and integral parts, with kernel $k(s,t)$ for $K$. The cited formula states
\[
 \norm{H}_w
 =\inf_{n\geq1}\esssup_t\left((|S|^*\one)(t)
       +\int_{\{|k(s,t)|\geq n\}}\abs{k(s,t)}\,d\nu(s)\right).
\]
For weakly compact $H$ the left side is zero. The nonnegative singular term must vanish, so $S=0$, and the essential suprema of the kernel tails tend to zero. The operators with kernels $k\one_{\{|k|<n\}}$ therefore converge to $H$ in norm. The operator-norm estimate needed here follows directly from Tonelli's theorem.
\end{proof}

\begin{lemma}\label{c:lem:weakrestriction}
For a finite Borel measure $\nu$ on a compact metrizable space and $H\in\W(E,L^1(\nu))$, the restriction $H|_{R}$ admits an approximable extension $C:E\to L^1(\nu)$.
\end{lemma}
\begin{proof}
Choose bounded-kernel operators $H_n$ with $\norm{H-H_n}\leq2^{-n}$, and put $H_0=0$. The series $\sum_{n\geq1}(H_n-H_{n-1})$ converges absolutely in operator norm to $H$.
For each $n$, Lemma~\ref{c:lem:boundedkernel} gives an approximable $C_n$ with
\[
 C_n|_{R}=(H_n-H_{n-1})|_{R},\qquad
 \norm{C_n}\leq\norm{H_n-H_{n-1}}+2^{-n}.
\]
The norm-convergent sum $C=\sum_nC_n$ is approximable, and its restriction to $R$ is $H|_{R}$.
\end{proof}

\begin{lemma}[A separable target for one operator]\label{c:lem:seprange}
If $H:E\to L^1(\mu)$ is bounded, its range is contained in a closed subspace $E_1$ lattice isometric to $L^1(\nu)$ for a standard finite measure $\nu$. The space $E_1$ may be chosen separable. If $H$ is weakly compact into $L^1(\mu)$, it is weakly compact into $E_1$.
\end{lemma}
\begin{proof}
The zero operator is harmless. Choose a sequence $(f_j)$ dense in $H(E)$ and set
\[
 h=\sum_{j\geq1}2^{-j}\frac{\abs{f_j}}{1+\norm{f_j}_1},\qquad
 A=\{h>0\},\qquad d\nu_0=h\,d\mu\text{ on }A.
\]
The measure $\nu_0$ is finite, and all $f_j$ vanish outside $A$ modulo null sets. Let $\Sigma_1$ be generated on $A$ by the scalar functions $f_j/h$. Multiplication by $h$, followed by zero extension, is a lattice isometry
\[
 L^1(A,\Sigma_1,\nu_0)\longrightarrow L^1(\mu)
\]
whose closed range $E_1$ contains every $f_j$.
Apply the bounded injective transform $u\mapsto (2/\pi)\arctan u$ to the functions $f_j/h$. The resulting countable map takes values in the compact metrizable cube $[-1,1]^{\N}$ and generates exactly $\Sigma_1$. Its pushforward is a finite Borel measure on that cube; pullback is onto $L^1(A,\Sigma_1,\nu_0)$ by density of simple functions from the generated sigma-algebra. This supplies the required compact-metric model and proves separability.
The weak topology of a closed subspace is the topology induced from the ambient weak topology by Hahn--Banach. A weak closure of a subset of $E_1$ stays in $E_1$, which proves the last assertion.
\end{proof}

\begin{proposition}[An initial approximable extension]\label{c:prop:initialextension}
Under the ambient condition \eqref{c:eq:ambient}, if $F\in\W(E_\mu,X_\mu)$, there exists $C_{\mathrm{init}}\in\Ap(E,X_\mu)$ with
\begin{equation}\label{c:eq:initialextension}
 C_{\mathrm{init}}|_{R}=(F\iota)|_{R}.
\end{equation}
\end{proposition}
\begin{proof}
The subspace $R_\mu$ is reflexive by Lemma~\ref{c:lem:rad}. Lift $F\iota:E\to E_\mu/R_\mu$ through $q_\mu$ by Lemma~\ref{c:lem:lifting}, obtaining a weakly compact $H:E\to E_\mu$. By Lemma~\ref{c:lem:seprange}, regard $H$ as an operator into a standard finite $L^1$-model $E_1$. Lemma~\ref{c:lem:weakrestriction} gives an approximable extension of $H|_{R}$ into $E_1$. Compose it with the inclusion $E_1\to E_\mu$ and $q_\mu$ to obtain $C_{\mathrm{init}}$.
\end{proof}

\subsection{Norm improvement and compact correction}

\begin{theorem}[Compact correction]\label{c:thm:maincorrection}
Assume the ambient condition \eqref{c:eq:ambient}. Suppose $G\in\B(E_\mu,X_\mu)$ and $F\in\W(E_\mu,X_\mu)$ satisfy $G|_{R_\mu}=F|_{R_\mu}$. For every $\varepsilon>0$ there exists $V\in\Ap(E_\mu,X_\mu)$ such that
\begin{equation}\label{c:eq:maincorrection}
 V|_{R_\mu}=G|_{R_\mu},\qquad \norm{V}\leq\norm{G}+\varepsilon.
\end{equation}
\end{theorem}

The next lemma is independent of any lattice structure in its target. This is important for the potentially nonseparable quotient $X_\mu$.

\begin{lemma}[Norm improvement from an approximable extension]\label{c:lem:normimprovement}
Let $Y$ be a Banach space, let $G_{\Delta}:E\to Y$ be bounded, and put $K=G_{\Delta}|_{R}$. If $K$ admits an approximable extension $C_{\mathrm{init}}:E\to Y$, then for every $\varepsilon>0$ it admits an approximable extension $V_{\Delta}$ with
\begin{equation}\label{c:eq:normimprovement}
 \norm{V_{\Delta}}\leq\norm{G_{\Delta}}+\varepsilon.
\end{equation}
\end{lemma}
\begin{proof}
Choose a finite-rank operator $A=\sum_{j=1}^k y_j\otimes\varphi_j$ with
$\norm{C_{\mathrm{init}}-A}<\varepsilon/4$, where $\varphi_j\in E^*$.
By \eqref{c:eq:shrinking},
\[
 \norm{\varphi_j|_{R}\circ(\Id_{R}-Q_N)}\longrightarrow0.
\]
Hahn--Banach extends each of these functionals to $\psi_{j,N}\in E^*$ with unchanged norm. Consequently
\[
 L_N=\sum_{j=1}^k y_j\otimes\psi_{j,N}
\]
is finite-rank,
\[
 L_N|_{R}=A|_{R}(\Id_{R}-Q_N),\qquad \norm{L_N}\longrightarrow0.
\]
The operator
\begin{equation}\label{c:eq:ZN}
 Z_N=(C_{\mathrm{init}}-A)(\Id_{E}-E_N)+L_N
\end{equation}
is approximable, and
\[
 Z_N|_{R}=K(\Id_{R}-Q_N),\qquad
 \norm{Z_N}\leq2\norm{C_{\mathrm{init}}-A}+\norm{L_N}.
\]
Choose $N$ so that the right side is less than $\varepsilon$. Now put
\[
 V_{\Delta}=G_{\Delta}E_N+Z_N.
\]
The first term is finite-rank, the second approximable, and
\[
 V_{\Delta}|_{R}=KQ_N+K(\Id_{R}-Q_N)=K,
 \qquad
 \norm{V_{\Delta}}\leq\norm{G_{\Delta}}+\varepsilon.
\]
\end{proof}

\begin{proof}[Proof of Theorem~\ref{c:thm:maincorrection}]
Apply Proposition~\ref{c:prop:initialextension} to $F$ and use $F|_{R_\mu}=G|_{R_\mu}$. This gives an approximable extension of $(G\iota)|_{R}$ from $E$ to $X_\mu$. Lemma~\ref{c:lem:normimprovement}, with $G_{\Delta}=G\iota$, produces $V_{\Delta}\in\Ap(E,X_\mu)$ satisfying
\[
 V_{\Delta}|_{R}=(G\iota)|_{R},\qquad
 \norm{V_{\Delta}}\leq\norm{G\iota}+\varepsilon\leq\norm{G}+\varepsilon.
\]
Set $V=V_{\Delta}P$. Then $V$ is approximable and $V\iota r=V_{\Delta}r=G\iota r$ for $r\in R$. Contractivity of $P$ proves the claimed norm bound.
\end{proof}

\subsection{Weak essential norms and an algebraic lifting}

\begin{lemma}\label{c:lem:pullbacknorm}
Under the ambient condition \eqref{c:eq:ambient}, for every $T\in\B(X_\mu)$,
\begin{equation}\label{c:eq:pullbacknorm}
 \norm{Tq_\mu}_{w,E_\mu\to X_\mu}=\norm{T^{\co}}.
\end{equation}
\end{lemma}
\begin{proof}
Lift $Tq_\mu$ to $A_T:E_\mu\to E_\mu$ with $q_\mu A_T=Tq_\mu$. By Lemma~\ref{c:lem:refquot}, $q_\mu^{\co}:E_\mu^{\co}\to X_\mu^{\co}$ is a surjective isometry, and
\[
 q_\mu^{\co}A_T^{\co}=T^{\co}q_\mu^{\co}.
\]
Thus $\norm{A_T^{\co}}=\norm{T^{\co}}$. By Corollary~\ref{c:cor:ALco},
\[
 \norm{Tq_\mu}_w\leq\norm{A_T}_w=\norm{A_T^{\co}}=\norm{T^{\co}}.
\]
The first inequality follows by composing weakly compact approximants with $q_\mu$. Conversely, \eqref{c:eq:coineq} and surjectivity of $q_\mu^{\co}$ give
$\norm{Tq_\mu}_w\geq\norm{T^{\co}q_\mu^{\co}}=\norm{T^{\co}}$.
\end{proof}

\begin{theorem}[The norm estimate for the ambient quotient]\label{c:thm:ambientnorm}
Under \eqref{c:eq:ambient}, every $T\in\B(X_\mu)$ satisfies
$\norm{T^{\co}}\leq\norm{T}_w\leq2\norm{T^{\co}}$.
\end{theorem}
\begin{proof}[Proof of Theorem~\ref{c:thm:ambientnorm} and Theorem~\ref{c:thm:mainnorm}]
Fix $T\in\B(X_\mu)$ and $\varepsilon>0$. Lemma~\ref{c:lem:pullbacknorm} provides $F\in\W(E_\mu,X_\mu)$ such that
\[
 \norm{Tq_\mu -F}\leq\norm{T^{\co}}+\varepsilon.
\]
Put $G=Tq_\mu -F$. Since $q_\mu |_{R_\mu}=0$, we have $G|_{R_\mu}=-F|_{R_\mu}$. Apply Theorem~\ref{c:thm:maincorrection} to $G$ and $-F$, obtaining $V\in\Ap(E_\mu,X_\mu)$ with
\[
 V|_{R_\mu}=G|_{R_\mu},\qquad \norm{V}\leq\norm{G}+\varepsilon.
\]
The weakly compact operator $F+V$ vanishes on $R_\mu$, so it factors as $Wq_\mu$ for a bounded $W:X_\mu\to X_\mu$. The operator $W$ is weakly compact: $Wq_\mu$ is weakly compact and $q_\mu$ is a metric quotient. For completeness, $B_{X_\mu}\subset q_\mu (2B_{E_\mu})$, so $W(B_{X_\mu})$ is contained in the relatively weakly compact set $(Wq_\mu)(2B_{E_\mu})$.
The norm identity $\norm{Aq_\mu}=\norm{A}$ for a metric quotient now yields
\[
 \begin{aligned}
 \norm{T}_w
 &\leq\norm{T-W}=\norm{(T-W)q_\mu}=\norm{G-V}\\
 &\leq2\norm{G}+\varepsilon
 \leq2\norm{T^{\co}}+3\varepsilon.
 \end{aligned}
\]
Let $\varepsilon\downarrow0$ and combine with \eqref{c:eq:coineq}. A bounded-below map from the Banach space $\Qw(X_\mu)$ has closed range, giving the final assertion.
\end{proof}

\begin{proposition}[A genuine algebra homomorphism]\label{c:prop:algebralift}
Choose any lifting $A_T:E_\mu\to E_\mu$ with $q_\mu A_T=Tq_\mu$. Then
\begin{equation}\label{c:eq:algebralift}
 \Lambda:\Qw(X_\mu)\longrightarrow\Qw(E_\mu),\qquad
 \Lambda([T])=[A_T]
\end{equation}
is a well-defined unital injective homomorphism satisfying
\begin{equation}\label{c:eq:lambdanorm}
 \norm{\Lambda([T])}=\norm{T^{\co}},\qquad
 \tfrac12\norm{[T]}\leq\norm{\Lambda([T])}\leq\norm{[T]}.
\end{equation}
\end{proposition}
\begin{proof}
Two liftings differ by an operator with range in the reflexive space $R_\mu$, hence by a weakly compact operator. A lifting of a weakly compact $Tq_\mu$ is weakly compact by Lemma~\ref{c:lem:lifting}. These observations prove independence both of the lifting and of the representative $T$.
Sums and scalar multiples of liftings lift the corresponding operators, and
\[
 q_\mu A_TA_U=Tq_\mu A_U=TUq_\mu.
\]
Thus $\Lambda$ is a homomorphism, with $A_{\Id_{X_\mu}}=\Id_{E_\mu}$ available. Its norm equality is proved in Lemma~\ref{c:lem:pullbacknorm}; injectivity and the two-sided bounds follow from Theorem~\ref{c:thm:ambientnorm}.
\end{proof}

\begin{remark}
The proposition is not by itself an inheritance theorem for incompressibility. An embedding of one algebra into an incompressible algebra does not, without further structure, prove incompressibility of its domain. The next subsection supplies the additional quantitative factorization.
\end{remark}

\subsection{Embedding, factorization, and incompressibility}

By Theorem~\ref{thm:main}, followed by its isometric target realization, there is a normalized embedding, henceforth denoted by $J:X\to F_\Delta=L^1(0,1)$, satisfying
\begin{equation}\label{c:eq:embinput}
 \norm{x}\leq\norm{Jx}\leq\kappa\norm{x},\qquad \kappa\coloneqq10C.
\end{equation}
Indeed, multiply the embedding of Theorem~\ref{thm:main} by $2C$.
This is the precise input from Section~\ref{sec:embedding} used in the algebraic argument.

\begin{lemma}[Extending the quotient embedding]\label{c:lem:globalembedding}
Under the ambient condition \eqref{c:eq:ambient}, the map
\begin{equation}\label{c:eq:globalembedding}
 D_\mu:X_\mu\longrightarrow E_\mu\oplus_1F_\Delta,\qquad
 D_\mu(q_\mu f)=\bigl((\Id_{E_\mu}-\iota P)f,\ J(qPf)\bigr)
\end{equation}
is well-defined and satisfies
\begin{equation}\label{c:eq:globalembeddingnorm}
 \norm{x}\leq\norm{D_\mu x}\leq(\kappa+2)\norm{x}.
\end{equation}
Its target is an $AL$-space.
\end{lemma}
\begin{proof}
Both components in \eqref{c:eq:globalembedding} vanish on $\iota R$. The quotient retraction \eqref{c:eq:quotretraction} gives
\[
 \norm{q_\mu f}
 \leq\norm{(\Id_{E_\mu}-\iota P)f}+\norm{qPf}
 \leq\norm{D_\mu(q_\mu f)}.
\]
Since $\norm{\Id_{E_\mu}-\iota P}\leq2$ and this operator annihilates $R_\mu$, its induced map from $X_\mu$ has norm at most $2$. Also $\norm{qPf}=\norm{\bar P(q_\mu f)}\leq\norm{q_\mu f}$. These inequalities prove the upper bound. Additivity of the norm on positive elements holds in an $\ell_1$-sum of $AL$-spaces.
\end{proof}

\begin{proposition}[Quantitative transfer criterion]\label{c:prop:criterion}
Assume the ambient condition \eqref{c:eq:ambient}, and suppose $D:X_\mu\to F$ embeds $X_\mu$ into an $AL$-space with
\[
 m\norm{x}\leq\norm{Dx}\leq M\norm{x}\qquad(x\in X_\mu).
\]
Then every bounded injective homomorphism $\Phi:\Qw(X_\mu)\to B$ into a normed algebra satisfies
\begin{equation}\label{c:eq:criterion}
 \norm{\Phi(a)}\geq\frac{m}{4M\norm{\Phi}^2}\norm{a}
 \qquad(a\in\Qw(X_\mu)).
\end{equation}
\end{proposition}
\begin{proof}
For $T\in\B(X_\mu)$ define
\[
 \delta_D(T)=\norm{DTq_\mu}_{w,E_\mu\to F}.
\]
Both $E_\mu$ and $F$ are $AL$-spaces. Hence Corollary~\ref{c:cor:ALco}, Lemma~\ref{c:lem:coembedding}, the isometry $q_\mu^{\co}$, and Theorem~\ref{c:thm:ambientnorm} imply
\begin{equation}\label{c:eq:deltabound}
 \begin{aligned}
 \delta_D(T)
 &=\norm{D^{\co}T^{\co}q_\mu^{\co}}\\
 &\geq\tfrac m2\norm{T^{\co}}
 \geq\tfrac m4\norm{T}_w.
 \end{aligned}
\end{equation}
There is nothing to prove if $[T]=0$. Otherwise $\delta_D(T)>0$. For $0<\eta<\delta_D(T)$, Theorem~\ref{c:thm:ALinput} gives
\[
 U:F\to\ell_1,\qquad V:\ell_1\to E_\mu,\qquad
 UDTq_\mu V=\Id_{\ell_1},\qquad
 \norm{U}\norm{V}<\frac1{\delta_D(T)-\eta}.
\]
Set $u=UD:X_\mu\to\ell_1$ and $v=q_\mu V:\ell_1\to X_\mu$. Then
\begin{equation}\label{c:eq:transferredfactorization}
 uTv=\Id_{\ell_1},\qquad
 \norm{u}\norm{v}<\frac{M}{\delta_D(T)-\eta}.
\end{equation}
Use $j,p,P_1$ from Lemma~\ref{c:lem:ellone}, and let $e=[P_1]\in\Qw(X_\mu)$. It is a nonzero idempotent, and
\[
 [ju]\,[T]\,[vp]=e.
\]
The element $\Phi(e)$ is a nonzero idempotent in $B$, so submultiplicativity gives $\norm{\Phi(e)}\geq1$. Therefore
\[
 1\leq\norm{\Phi}^2\norm{u}\norm{v}\norm{\Phi([T])}.
\]
Using \eqref{c:eq:transferredfactorization} and letting $\eta\downarrow0$, we obtain
\[
 \norm{\Phi([T])}\geq\frac{\delta_D(T)}{M\norm{\Phi}^2}.
\]
Equation~\eqref{c:eq:deltabound} completes the proof.
\end{proof}

\begin{theorem}[Incompressibility for the ambient quotient]\label{c:thm:ambientalgebra}
Assume \eqref{c:eq:ambient}, with $\kappa$ as in \eqref{c:eq:embinput}.
For every normed algebra $B$ and bounded injective homomorphism $\Phi:\Qw(X_\mu)\to B$,
\[
 \norm{\Phi(a)}\geq
 \frac{\norm{a}}{4(\kappa+2)\norm{\Phi}^2}\qquad(a\in\Qw(X_\mu)).
\]
For $X$ the denominator improves to $4\kappa\norm{\Phi}^2$.
\end{theorem}
\begin{proof}[Proof of Theorem~\ref{c:thm:ambientalgebra} and Theorem~\ref{c:thm:mainalgebra}]
Apply Proposition~\ref{c:prop:criterion} to $D_\mu$ from Lemma~\ref{c:lem:globalembedding}, with $m=1$ and $M=\kappa+2$. For $X_\mu=X$ use $D=J$ directly, with $m=1$ and $M=\kappa$. For a contractive $\Phi$, the resulting lower bound is at least $1/[4(\kappa+2)]$, respectively $1/(4\kappa)$.
\end{proof}

\begin{remark}[Why the sandwich need not be multiplicative]\label{c:rem:sandwich}
The argument transfers factorizations through $DTq_\mu$; it does not claim that $T\mapsto DTq_\mu$ is an algebra homomorphism. Even when $F=E_\mu$,
\[
 (DTq_\mu)(DUq_\mu)=DT(q_\mu D)Uq_\mu,
\]
which need not be $DTUq_\mu$. Nor is an arbitrary weakly compact approximation to $DTq_\mu$ assumed to annihilate $R_\mu$ or have range in $D(X_\mu)$. The compact correction theorem resolves the needed approximation problem before the sandwich is used.
\end{remark}

\subsection*{Acknowledgement}
The author used ChatGPT to assist with proof development, reference checking, and manuscript preparation.

\end{document}